\documentclass[11pt, reqno]{amsart}
\usepackage[margin=1in]{geometry}
\usepackage{amsmath, amssymb, amsthm,bbm,bm,mathtools}
\usepackage[shortlabels]{enumitem}
\usepackage{url}
\usepackage{xcolor}

\usepackage[hidelinks]{hyperref}

\usepackage[noabbrev,capitalize,nameinlink]{cleveref}
\crefformat{equation}{(#2#1#3)}
\crefrangeformat{equation}{(#3#1#4) to~(#5#2#6)}
\crefmultiformat{equation}{(#2#1#3)}
{ and~(#2#1#3)}{, (#2#1#3)}{ and~(#2#1#3)}

\newtheorem{theorem}{Theorem}
\newtheorem{lemma}[theorem]{Lemma}
\newtheorem{corollary}[theorem]{Corollary} 
\newtheorem{proposition}[theorem]{Proposition}

\newtheorem{conjecture}[theorem]{Conjecture}

\theoremstyle{definition}
\newtheorem{definition}[theorem]{Definition}

\theoremstyle{remark}
\newtheorem{remark}[theorem]{Remark}

\numberwithin{theorem}{section}
\numberwithin{equation}{section}

\newcommand{\abs}[1]{\left\lvert#1\right\rvert}
\newcommand{\abss}[1]{\lvert#1\rvert}
\newcommand{\norm}[1]{\left\lVert#1\right\rVert}

\newcommand{\floor}[1]{\left\lfloor #1 \right\rfloor}
\newcommand{\ceil}[1]{\left\lceil #1 \right\rceil}
\newcommand{\paren}[1]{\left( #1 \right)}
\newcommand{\set}[1]{\left\{ #1 \right\}}

\newcommand{\wh}{\widehat}
\newcommand{\wt}{\widetilde}
\renewcommand{\epsilon}{\varepsilon}

\newcommand{\CC}{\mathbb{C}}
\newcommand{\EE}{\mathbb{E}}

\newcommand{\RR}{\mathbb{R}}
\newcommand{\PP}{\mathbb{P}}
\newcommand{\TT}{\mathbb{T}}

\newcommand{\ZZ}{\mathbb{Z}}

\newcommand{\bfa}{\mathbf{a}}

\newcommand{\C}{\mathbb C}

\newcommand{\E}{\mathbb E}

\newcommand{\R}{\mathbb R}
\newcommand{\T}{\mathbb T}
\newcommand{\Z}{\mathbb Z}

\title{Uniform sets with few progressions via colorings}
\author[Deng]{MingYang Deng}
\author[Tidor]{Jonathan Tidor}
\author[Zhao]{Yufei Zhao}
\thanks{Tidor was supported by a Stanford Science Fellowship. Zhao was supported by NSF CAREER Award DMS-2044606 and a Sloan Research Fellowship.}

\address{Deng, Zhao: Massachusetts Institute of Technology, Cambridge, MA, USA}
\email{\textnormal{\{}dengm,yufeiz\textnormal{\}}@mit.edu}

\address{Tidor: Stanford University, Stanford, CA, USA}
\email{jtidor@stanford.edu}

\begin{document}

\begin{abstract}
Ruzsa asked whether there exist Fourier-uniform subsets of $\mathbb Z/N\mathbb Z$ with density $\alpha$ and 4-term arithmetic progression (4-AP) density at most $\alpha^C$, for arbitrarily large $C$. Gowers constructed Fourier uniform sets with density $\alpha$ and 4-AP density at most $\alpha^{4+c}$ for some small constant $c>0$. We show that an affirmative answer to Ruzsa's question would follow from the existence of an $N^{o(1)}$-coloring of $[N]$ without symmetrically colored 4-APs. For a broad and natural class of constructions of Fourier-uniform subsets of $\mathbb Z/N\mathbb Z$, we show that Ruzsa's question is equivalent to our arithmetic Ramsey question.

We prove analogous results for all even-length APs. For each odd $k\geq 5$, we show that there exist $U^{k-2}$-uniform subsets of $\mathbb Z/N\mathbb Z$ with density $\alpha$ and $k$-AP density at most $\alpha^{c_k \log(1/\alpha)}$. We also prove generalizations to arbitrary one-dimensional patterns.
\end{abstract}

\maketitle

\section{Introduction}

A basic and central fact in additive combinatorics is that the density 3-term arithmetic progressions (3-APs) in a Fourier uniform subset of $\ZZ/N\ZZ$ of density $\alpha$ is close to $\alpha^3$ (the random estimate).
This fact plays a key role in the proof of Roth's theorem on the existence of 3-APs in dense subsets of the integers \cite{Rot53}. 
Also important is that Fourier uniformity does not control 4-AP counts---a now-standard construction gives Fourier-uniform sets of density $\alpha + o(1)$ and 4-AP density much higher than the random estimate of $\alpha^4$: $\{n \in \ZZ/N\ZZ : 0 \le n^2 \pmod N < \alpha N\}$. 
This construction was given by Gowers~\cite{Gow01} in his groundbreaking work giving a new proof of Szemer\'edi's theorem that led to the development of higher order Fourier analysis.
On the other hand, is there a Fourier uniform set with very few 4-APs? This is the question that we study in this paper.
Interestingly, the above construction does not easily suggest any modifications with too few 4-APs.

\subsection{Fourier uniform sets with few 4-APs}

We say that $A \subset \ZZ/N\ZZ$ is \emph{$\epsilon$-Fourier uniform} if $\abss{\wh{1_A}(r)} \le \epsilon$ for all $r \ne 0$. Here the Fourier transform is defined as, for $f \colon \ZZ/N\ZZ \to \CC$ and $r \in \ZZ/N\ZZ$,
\[
\wh{f} (r) \coloneqq \EE_{n \in \ZZ/N\ZZ} f(n) e^{-2\pi i rn/N}.
\]

In a finite abelian group $G$ (typically $\ZZ/N\ZZ$ for us), for any $k \ge 3$ and functions $f_1, \dots, f_k \colon G \to \CC$, define
\[
\Lambda_k(f_1, \dots, f_k) \coloneqq \EE_{n, d} f_1(n) f_2(n+d) \cdots f_k(n+(k-1)d).
\]
Also 
\[
\Lambda_k(f) \coloneqq \Lambda_k(f, f, \dots, f) = \EE_{n, d} f(n) f(n+d) \cdots f(n+(k-1)d).
\]
Given $A \subset \ZZ/N\ZZ$, we call $\Lambda_k(1_A)$ the \emph{$k$-AP density} of $A$. It is the number of $k$-APs in $A$ divided by $N^2$, where we count each $k$-AP twice (forward and backward) and each trivial $k$-AP (with common difference zero) once.

Gowers initially conjectured that a Fourier uniform subset of $\ZZ/N\ZZ$ of density $\alpha$ should have 4-AP density at least $\alpha^4 +o(1)$. 

\begin{conjecture}[{Gowers \cite[Conjecture 4.1]{Gow01}}]
\label{conj:gowers4ap}
For every $\epsilon > 0$, there exists $\eta > 0$ so that any $\eta$-Fourier uniform $A \subset \ZZ/N\ZZ$ with $\abs{A} \ge  \alpha N$ has $\Lambda_4(1_A) \ge \alpha^4 - \epsilon$. 
\end{conjecture}

Gowers subsequently disproved his conjecture by constructing a counterexample with 4-AP density less than $\alpha^{4+c}$ for some absolute constant $c > 0$.

\begin{theorem}[{Gowers \cite[Theorem 6]{Gow20}}]
There exist some constant $c>0$ and sets $A_N \subset \ZZ/N\ZZ$ so that 
$\abs{A_N}/N \ge 1/2 + o(1)$, 
$A_N$ is $o(1)$-Fourier uniform, 
and $\Lambda_4(1_{A_N}) \le 2^{-4} - c$.
\end{theorem}

It will be convenient to introduce the following definition.

\begin{definition}[Minimum 4-AP density in a Fourier uniform set]
For any $0 < \alpha < 1$, 
let $\rho_4(\alpha)$ be the largest real number so that for any $\epsilon > 0$ there exist $\eta > 0 $ and $N_0$ so that every $\eta$-Fourier uniform $A \subset \ZZ/N\ZZ$ with $N \ge N_0$ and $\abs{A} \ge (\alpha-\epsilon) N$ has $\Lambda_4(1_A) \ge \rho_4(\alpha) - \epsilon$.
\end{definition}

Gowers' construction~\cite{Gow20} shows that $\rho_4(1/2)\leq 2^{-4}-2^{-30}<(1/2)^{4+c}$ for $c\approx2\times 10^{-8}$. 
Via a tensor power trick, this implies $\rho_4(\alpha)<\alpha^{4+c}$ for all sufficiently small $\alpha$. Wolf~\cite{Wol10} modified this construction to obtain $\rho_4(1/2)\leq 2^{-4}-\tfrac{2\times 36}{9(5\times 18)^22^{12}}<(1/2)^{4+c}$ for $c\approx5\times 10^{-6}$. Again this implies that $\rho_4(\alpha)<\alpha^{4+c}$ for all sufficiently small $\alpha$. In \cref{cor:quant-bounds-4AP} below of our main results, we improve these bounds to $\rho_4(\alpha)<\alpha^{4.4}$ for all sufficiently small $\alpha > 0$.

Ruzsa~\cite[Problem 3.2]{CL07} asked whether $\rho_4(\alpha)$ goes to zero faster than any polynomial in $\alpha$. 
We believe that the answer should be yes, and we state it as a conjecture.
One reason to believe this conjecture is that it would be consistent with the intuition that Fourier uniformity does not control 4-AP counts.

\begin{conjecture}[Ruzsa's question] \label{conj:unif4ap}
For every $C >0$, one has $\rho_4(\alpha) < \alpha^C$ for all sufficiently small $\alpha$.
\end{conjecture}

\begin{remark}
Current upper bounds~\cite{GT17} on the maximum size of a 4-AP-free subset of $[N]$ (i.e., Szemer\'edi's theorem for 4-APs) imply $\rho_4(\alpha) \ge \exp(-\alpha^{-\Omega(1)})$, even if we forget the uniformity requirement in the definition of $\rho_4$. We do not know any better lower bounds.
Conversely, each lower bound on $\rho_4(\alpha)$, together with the standard density increment proof strategy of Roth's theorem, would imply some upper bound on Szemer\'edi's theorem for 4-APs.
\end{remark} 

We pose a conjecture in arithmetic Ramsey theory and show that it would imply \cref{conj:unif4ap}.
Here a $4$-AP $n, n+d, n+2d, n+3d$ with $d \ne 0$ is said to be \emph{symmetrically colored} if $n$ and $n+3d$ have the same color, and $n+d$ and $n+2d$ have the same color.

\begin{conjecture}[Avoiding symmetrically colored 4-APs] \label{conj:sc4ap}
For all $N$, 
there exists an $N^{o(1)}$-coloring of $[N]$ without symmetrically colored 4-APs. 
\end{conjecture}

\begin{theorem}
\label{thm:even-color-non-quan}
\cref{conj:sc4ap} implies \cref{conj:unif4ap}.
\end{theorem}

We were unable to resolve \cref{conj:sc4ap} despite some effort.
Here is why we think \cref{conj:sc4ap} might be hard. In a coloring of $[N]$ without symmetrically colored 4-APs, each color class must be 4-AP-free. Essentially all known constructions of large 4-AP-free sets are related to the famous Behrend construction~\cite{Beh32}, coming from high-dimensional convex bodies. However, it seems difficult to avoid symmetrically colored 4-APs by using Behrend-like sets as color classes.

We prove the following quantitative version of \cref{thm:even-color-non-quan}.

\begin{theorem} \label{thm:even-color-quan}
Suppose there exists an $r$-coloring of $\ZZ/N\ZZ$ avoiding symmetrically colored 4-APs.
Then there exists a constant $C$ (depending on $r, N$), such that for all $0< \alpha < 1/2$,
\[
	\rho_4(\alpha) \le C \alpha^{3 + (1/2) \log_r N}.
\]
\end{theorem}

Here is a 3-coloring of $\Z/22\Z$ that avoids symmetrically colored 4-APs (found by computer search):
\[
1333221232131211333233.
\]
This implies the following bound, improving earlier bounds by Gowers~\cite{Gow20} and Wolf~\cite{Wol10}.

\begin{corollary}
\label{cor:quant-bounds-4AP}
$\rho_4(\alpha) = O(\alpha^{3+(1/2)\log_3 22}) = O(\alpha^{4.406})$.
\end{corollary}

Here is a family of constructions of Fourier uniform sets in $\ZZ/N\ZZ$. 
Fix some Jordan measurable $S \subset [0,1)^2$ (here $S$ being \emph{Jordan measurable} is the same as saying that the boundary of $S$ has zero measure, i.e., $1_S$ is Riemann integrable) such that the intersection of $S$ with every vertical line with $x$-coordinate in $[0,1)$ has measure $\alpha$.
Set $A_N \subset \ZZ/N\ZZ$ to be
\[
A_N = \{ n \in \ZZ/N\ZZ \colon  (n/N \bmod 1, n^2/N \bmod 1) \in S\}.
\]
As we will see in \cref{sec:torus-formalism}, $A_N$ is $o(1)$-Fourier uniform.
This construction is quite broad and natural from the perspective of quadratic Fourier analysis.
For example, Gowers' construction~\cite{Gow20} is a special case.
We will show that \cref{conj:unif4ap}, when restricted to constructions of subsets of $\ZZ/N\ZZ$ of this form, is actually equivalent to \cref{conj:sc4ap}. We will make the statement of this converse precise in \cref{sec:torus-formalism}.

\subsection{\texorpdfstring{$U^{k-2}$}{Uk-2}-uniform sets with few \texorpdfstring{$k$}{k}-APs}

It is now a standard fact that the Gowers uniformity norms~\cite{Gow01} control longer APs: a $U^{k-1}$-uniform set has $k$-AP density close to random.
So a natural extension of Ruzsa's $4$-AP question is whether there are $U^{k-2}$-uniform set with few $k$-APs.

Recall that for a function $f\colon G\to\C$ on a finite abelian group $G$, the \emph{$k$-th Gowers uniformity norm} is defined by 
\[\|f\|_{U^k}^{2^k}=\E_{n,h_1,\ldots,h_k\in G}\partial_{h_1}\partial_{h_2}\cdots\partial_{h_k}f(n),\]
where $\partial_hf(n)=f(n)\overline{f(n-h)}$.
We say that a sequence of sets $A_N \subset \Z/N\Z$ is \emph{$U^k$-uniform} if $\norm{1_{A_N} - \abs{A_N}/N}_{U^k} = o(1)$.
It is a standard fact that Fourier uniformity is equivalent to $U^2$ uniformity.

Gowers made the following generalization of \cref{conj:unif4ap} to all even APs.

\begin{conjecture}[{Gowers \cite[Conjecture 4.2]{Gow01}}]
For every even $k\ge 4$ and $\epsilon > 0$, there exists $\eta > 0$ so that if $A \subset \ZZ/N\ZZ$ satisfies $\norm{1_A - \alpha}_{U^{k-2}} \le \eta$ where $\alpha = \abs{A}/N$, then $\Lambda_k(1_A) \ge \alpha^k - \epsilon$. 
\end{conjecture}

As mentioned above, this conjecture has been disproved for $k=4$. In light of this, it is reasonable to believe that it is false for all $k\ge 4$.
Furthermore, we conjecture (precise statement to follow) that for each $k \ge 4$, there exist $U^{k-2}$-uniform sets with density $\alpha$ and $k$-AP density less than $\alpha^C$, for arbitrarily large $C$.
We prove our conjecture for odd\footnote{Gowers~\cite[p.~488]{Gow01} writes that ``It can be shown that quadratically uniform sets sometimes contain significantly fewer progressions of length five than random sets of the same cardinality. However, the example depends in an essential way on 5 being odd,'' though without providing further details.} $k$, and reduce the conjecture for each even $k \ge 4$ to a conjecture about avoiding symmetrically colored $k$-APs.

\begin{definition}[Minimum $k$-AP density in a $U^{k-2}$-uniform set]
For any integer $k \ge 4$ and $0 < \alpha < 1$, 
let $\rho_k(\alpha)$ be the largest real number so that 
for any $\epsilon > 0$, there exist $\eta > 0$ and $N_0$ so that if $N \ge N_0$ and $A \subset \ZZ/N\ZZ$ satisfies $\norm{1_A - \alpha}_{U^{k-2}} \le \eta$, then
$\Lambda_k(1_A) \ge \rho_k(\alpha) - \epsilon$.
\end{definition}

\begin{conjecture}[Ruzsa's question extended to $k$-APs] 
\label{conj:unifkap}
	Let $k \ge 4$. 
	For every $C>0$, $\rho_k(\alpha) < \alpha^C$ for all sufficiently small $\alpha$.
\end{conjecture}

We prove \cref{conj:unifkap} for odd $k$.

\begin{theorem}[Odd $k$]
\label{thm:kap-cons-odd}
For every odd integer $k \ge 5$ there is some constant $c_k>0$ such that $\rho_k(\alpha) < \alpha^{c_k \log(1/\alpha)}$ for all $0 < \alpha < 1/2$.
\end{theorem}

Ruzsa \cite{BHK05} constructed Fourier-uniform sets with density $\alpha$ and 5-AP density $\alpha^{c\log(1/\alpha)}$.  The above result is stronger, even for $k=5$, since the provided sets are $U^{k-2}$-uniform.

For even $k\geq 4$, a \emph{symmetrically colored $k$-AP} is some $n, n+d, \cdots, n+(k-1)d$ with $d \ne 0$ and where $n + (i-1)d$ and $n + (k-i)d$ have the same color for all $i = 1, \dots, k/2$.

\begin{conjecture}[Avoiding symmetrically colored $k$-APs]
\label{conj:sckap}
Let $k \ge 4$ be even. 
For all $N$, there exists an $N^{o(1)}$-coloring of $[N]$ avoiding symmetrically colored $k$-APs.
\end{conjecture}

Choosing the coloring uniformly at random, an application of the Lov\'asz local lemma implies that there is an $O(N^{2/k})$-coloring of $[N]$ avoiding symmetrically colored $k$-APs. On the other hand, any such coloring must also avoid monochromatic $k$-APs and thus must use at least $\Omega(\exp((\log \log N)^{c_k}))$ colors, by Leng--Sah--Sawhney's bound on van der Waerden's theorem \cite{LSS24}. These are the best bounds that we know of for general $k$.

Also note that the $k=4$ case (i.e., \cref{conj:sc4ap}) implies \cref{conj:sckap} for all $k$, since each symmetrically colored $k$-AP contains a symmetrically colored $4$-AP.

\begin{theorem}[Even $k$]
\label{thm:kap-reduction-even}
\cref{conj:sckap} for each even $k \ge 4$ implies \cref{conj:unifkap} for the same $k$.
\end{theorem}

\subsection{General patterns}

Our results also generalize from $k$-APs to the setting of arbitrary one-dimensional $k$-point patterns.
The proofs are mostly similar to the $k$-AP case, though there are some additional technical difficulties.

\begin{definition}
Let $\bfa = (a_1, \dots, a_k)$ have increasing integer coordinates $a_1 < a_2 < \cdots < a_k$. We define an \emph{$\bfa$-AP} to be sequence of the form $n + a_1 d, \ldots, n + a_k d$ for some $n, d$, and we say that this $\bfa$-AP is \emph{non-trivial} if $d \ne 0$.	

    In a finite abelian group $G$, given functions $f_1, \dots, f_k \colon G \to \CC$, define
\[
\Lambda_{\bfa}(f_1, \dots, f_k) \coloneqq \EE_{n, d} f_1(n+a_1d) f_2(n+a_2d) \cdots f_k(n+a_kd).
\]
\end{definition}

Then a $k$-AP is the same as an $\bfa$-AP with $\bfa = (0, 1, \dots, k-1)$ and $\Lambda_k$ agrees with $\Lambda_{\bfa}$ in this case.

\begin{definition} [Minimum $\bfa$-AP density in a $U^{k-2}$-uniform set]
Let $\bfa = (a_1, \dots, a_k)$ have increasing integer coordinates.
For $0 < \alpha < 1$, 
let $\rho_\bfa(\alpha)$ be the largest real number so that 
for any $\epsilon > 0$ there exist $\eta > 0$ and $N_0$ so that if $N \ge N_0$ and $A \subset \ZZ/N\ZZ$ satisfies $\norm{1_A - \alpha}_{U^{k-2}} \le \eta$, then
$\Lambda_\bfa(1_A) \ge \rho_k(\alpha) - \epsilon$.
\end{definition}

\begin{conjecture}[Ruzsa's question extended to $\bfa$-APs] 
\label{conj:unifaap}
Let $k \ge 4$ and $\bfa = (a_1, \dots, a_k)$ have increasing integer coordinates.
Then for every $C > 0$, one has $\rho_\bfa(\alpha) < \alpha^C$ for all sufficiently small $\alpha > 0$.	
\end{conjecture}

\begin{definition}
	We say that $\bfa = (a_1, \dots, a_k)$ is \emph{symmetric} if $k$ is even and $a_1 + a_k = a_2 + a_{k-1} = a_3 + a_{k-2} = \cdots$.
\end{definition}

When $k=4$ and $\bfa$ is a symmetric, an $\bfa$-AP corresponds to what is sometimes referred to as a ``parallelogram pattern.''
We prove \cref{conj:unifaap} in all cases except when $k$ is even and $\bfa$ is symmetric. (Although $k$ being even is already part of the definition of $\bfa$ being symmetric, we still remind the reader of the parity of $k$ for emphasis.)

\begin{theorem} \label{thm:general_pattern}
	Let $\bfa = (a_1, \dots, a_k)$ have increasing integer coordinates and $k \ge 4$.
	Suppose either $k$ is odd or $\bfa$ is not symmetric.
	Then there exists a constant $c > 0$ (depending on $\bfa$ only) such that $\rho_\bfa(\alpha) < \alpha^{c \log(1/\alpha)}$ for all $0 < \alpha < 1/2$.
\end{theorem}

When $k$ is even and $\bfa$ is symmetric, our results are analogous to those for $k$-APs with $k$ even, except that instead of finding an $N^{o(1)}$-coloring of $[N]$ avoiding symmetrically colored $k$-APs, we need an $N^{o(1)}$-coloring of $[N]$ avoiding symmetrically colored $\bfa$-APs. Here a \emph{symmetrically colored $\bfa$-AP} is some $n + a_1 d, \dots, n + a_k d$ where $n+a_id$ and $n+a_{k-i+1}d$ have the same color for each $i =1, \dots, k/2$.

\begin{conjecture}[Avoiding symmmetrically colored $\bfa$-APs] \label{conj:general-coloring}
Let $k \ge 4$ be even.
Let $\bfa = (a_1, \dots, a_k)$ be symmetric with increasing integer coordinates.
Then there exists an $N^{o(1)}$-coloring of $[N]$ without symmetrically colored $\bfa$-APs.
\end{conjecture}

Note that \cref{conj:general-coloring} for $k=4$ implies the conjecture for all even $k \ge 4$. Here is our most general reduction result.

\begin{theorem} \label{thm:general_reduction}
Let $k \ge 4$ be even.
Let $\bfa = (a_1, \dots, a_k)$ be symmetric with increasing integer coordinates.
Then \cref{conj:general-coloring} for this $\bfa$ implies \cref{conj:unifaap} for the same $\bfa$.
\end{theorem}

\vspace{0.5em}
\noindent\textbf{Acknowledgements:} We thank Zach Hunter for helpful discussions about \cref{lem:power} and colorings without symmetrically colored 4-APs. We thank Josef Greilhuber for helpful discussions around the constructions in \cref{rem:asym-pairing}.

\section{Constructing uniform sets}
\label{sec:torus-formalism}

As is standard in additive combinatorics, there is no often substantive difference between considering subsets $A\subseteq \ZZ/N\ZZ$ versus functions $f \colon \Z/N\Z\to [0,1]$ for large $N$, as $f$ can be turned into $A$ via sampling, and all relevant quantities concentrate accordingly, such as $\Lambda_k(1_A) \approx \Lambda_k(f)$. 
We will primarily work with functions from now on.

Let $\TT \coloneqq \RR/\ZZ$ be the circle. 
We say that a Riemann integrable function $F$ on $\TT^2$ has \emph{constant first marginal $\alpha$} if $\int_\TT  F(x,y)\,\text dy$ is the constant $\alpha$ for almost every $x$.
Starting with some fixed Riemann integrable $F \colon \TT \to [0,1]$, consider $f_N \colon \ZZ/N\ZZ \to [0,1]$ defined by 
\begin{equation} \label{eq:fN}
f_N(n) = F(n/N \bmod 1, n^{k-2}/N \bmod 1).
\end{equation}
This can then be turned into a subset $A_N \subset \Z/N\Z$ by sampling if one wishes.
We can deduce the following asymptotics using standard Weyl polynomial equidistribution arguments (see \cref{sec:equi-conv} for proof sketches).
They tell us that our sequence is $U^{k-2}$-uniform and has $k$-AP density approaching $\wt \Lambda_k(F)$ defined below. 
This gives a broad and natural class of constructions for Ruzsa's question and its extensions.\footnote{Previous constructions of Gowers~\cite{Gow20} and Wolf~\cite{Wol10} showing $\rho_4(1/2)<1/16$ have this form with $F(x,y)=\frac12+\frac18\left(e^{6\pi i y}+e^{2\pi i y}+e^{-2\pi i y}+e^{-6\pi iy}\right)f(x)$ for some explicit $f\colon\TT\to\{-1,0,1\}$.}

\begin{proposition}[4-AP special case]
Given a Riemann integrable $F \colon \TT^2 \to [0,1]$,
define $f_N$ as in \cref{eq:fN}. If $F$ has constant first marginal $\alpha$, then, as $N \to \infty$,
\[
\norm{f_N - \alpha}_{U^2} \to 0
\]
and
\[
\Lambda_4(f_N) 
\to 
\EE_{\substack{x_1, x_2, x_3, x_4, y_1, y_2, y_4, y_4 \in \TT: \\ x_2 - x_1 = x_3 - x_2 = x_4 - x_3 \\ \text{ and } y_1 - 3y_2 + 3y_3 - y_4=0}}  
F(x_1, y_1)
F(x_2, y_2)
F(x_3, y_3)
F(x_4, y_4).
\]
\end{proposition} 

\begin{proposition}\label{prop:conv}
For $k\geq 4$, given a Riemann integrable $F \colon \TT^2 \to [0,1]$,
define $f_N$ as in \cref{eq:fN}.
If $F$ has constant first marginal $\alpha$, then, as $N \to \infty$, 
\[
\norm{f_N - \alpha}_{U^{k-2}} \to 0
\]
and
\[
\Lambda_k(f_N) \to \wt\Lambda_k(F),
\]
where
\[
\wt\Lambda_k(F) \coloneqq \EE_{(x_1, \dots, x_k, y_1, \dots, y_k) \in V} F(x_1, y_1) \cdots F(x_k, y_k)
\]
where $V$ is the subset of $\TT^{2k}$ defined by all points $(x_1, \dots, x_k, y_1, \dots, y_k)$ satisfying
\[
x_2 - x_1 = \cdots = x_k - x_{k-1}
\]
and
\[
\sum_{i=0}^{k-1} (-1)^i\binom{k-1}{i} y_i = 0.
\]
\end{proposition} 

The following definition captures the minimum possible $k$-AP density that can be obtained via such constructions.

\begin{definition} 
Let $\wt \rho_k(\alpha)$ denote the infimum of $\wt\Lambda_k(F)$ ranging over all Riemann integrable $F \colon \TT^2 \to [0,1]$ with constant first marginal $\alpha$.
\end{definition}

\begin{proposition}
\label{prop:set-torus-bound}
	$\rho_k(\alpha) \le \wt\rho_k(\alpha)$ for all $k$ and $\alpha$.
\end{proposition}

\begin{conjecture}[Ruzsa's question restricted to above constructions]
\label{conj:wtunifkap}
	Let $k \ge 4$. 
	For every $C>0$, $\wt\rho_k(\alpha) < \alpha^C$ for all sufficiently small $\alpha$.
\end{conjecture}

Note that $\wt\rho_k$ is monotonic, since if $\alpha<\alpha'$ and $F\colon\TT^2\to[0,1]$ has constant first marginal $\alpha'$, we see that $\tfrac{\alpha}{\alpha'}F$ has constant first marginal $\alpha$ and satisfies $\wt\Lambda_k(\tfrac{\alpha}{\alpha'}F)=\tfrac{\alpha^k}{\alpha'^k}\wt\Lambda_k(F)\ge \wt\Lambda_k(F)$.

Now we state our main results in terms of $\wt\rho_k$.

\begin{theorem}[Odd $k$] \label{thm:odd-torus}
	Let $k \ge 5$ be odd. There is some constant $c_k > 0$ so that for all $0 < \alpha < 1/2$,
	\[
	\wt\rho_k(\alpha) \le \alpha^{c_k \log(1/\alpha)}.
	\]
\end{theorem}

\begin{theorem}[$k=4$] \label{thm:4-torus}
Suppose there exists an $r$-coloring of $\ZZ/N\ZZ$ avoiding symmetrically colored 4-APs.
There exists a constant $C$ (depending on $r, N$), such that for all $0< \alpha < 1$,
\[
	\wt\rho_k(\alpha) \le C \alpha^{3 + (1/2) \log_r N}.
\]
\end{theorem}

\begin{theorem}[Even $k$] \label{thm:even-color-torus}
Let $k \ge 4$ be even.
Suppose there exists an $r$-coloring of $\ZZ/N\ZZ$ avoiding symmetrically colored $k$-APs.
For every $\epsilon>0$, there exists a constant $C$ (depending on $k, r, N,\epsilon$), such that for all $0< \alpha < 1$,
\[
	\wt\rho_k(\alpha) \le C \alpha^{k-1 + (\log_r N)/(k-1) - \epsilon}.
\]
\end{theorem}

For even $k \ge 6$, our quantitative bounds are slightly weaker than for $k=4$ (the main bottleneck is that we do not know a generalization of \cref{lem:13avoid} to larger $k$).

Together with \cref{prop:set-torus-bound}, \cref{thm:odd-torus,thm:4-torus,thm:even-color-torus} imply \cref{thm:kap-cons-odd,thm:even-color-quan,thm:kap-reduction-even}. (The deduction of \cref{thm:kap-reduction-even} follows from the observation that if there exists an $N^{\epsilon}$-coloring of $[N]$ avoiding symmetrically colored $k$-APs, taking $k$ copies of this coloring gives a $kN^{\epsilon}$-coloring of $\Z/kN\Z$ that avoids symmetrically colored $k$-APs.) We prove \cref{thm:odd-torus} in \cref{sec:odd-torus-const} and \cref{thm:4-torus,thm:even-color-torus} in \cref{sec:even-torus-const}. The constructions in these sections need some building blocks that we collect in \cref{sec:non-trivial-sol}. 

We also prove the following converse to \cref{thm:even-color-torus}. We show that given a function $F\colon\TT^2\to[0,1]$ with constant first marginal and $\wt\Lambda_k(F)$ small, one can extract a coloring of $\Z/N\Z$ with few colors that avoids symmetrically colored $k$-APs. See \cref{sec:extract} for proof.

\begin{theorem}[Even $k$ converse]
\label{thm:even-k-converse}
Let $k\geq 4$ be even. Suppose there exist $C, \alpha_0 > 0$ such that  such that $\wt\rho_k(\alpha)\le \alpha^C$ for all $0 < \alpha < \alpha_0$.
Then for every $\gamma > 1/(C/(2k) - 3k/4)$, for all sufficiently large $N$, 
there exists an $\ceil{N^\gamma}$-coloring of $\Z/N\Z$ avoiding symmetrically colored $k$-APs.
\end{theorem}

Recall that \cref{conj:sckap} says that there exist $N^{o(1)}$-colorings of $[N]$ avoiding symmetrically colored $k$-APs. \cref{thm:even-color-torus,thm:even-k-converse} together imply the following equivalence of conjectures.

\begin{corollary} \label{cor:k-ap-equiv}
    For each even $k\ge 4$, \cref{conj:wtunifkap} for this $k$ is equivalent to \cref{conj:sckap} for the same $k$.
\end{corollary}

\section{Some building blocks}
\label{sec:non-trivial-sol}

In our constructions we will need colorings that avoid various configurations. We start by recording a fairly standard fact that a coloring based on Behrend sets avoids many small configurations.

\begin{definition}
A $k$-pattern is a triple $(n_1,n_2,n_3)$, not all equal, satisfying $a n_1 + b n_2 = (a+b) n_3$ for some positive integers $a,b$ and $a + b \le k-1$. 
\end{definition}

\begin{lemma}\label{lem:avoid-pattern}
 For every $k$ there is some $C$ so that for every $N$, there is an $r$-coloring of $\Z/N\Z$ with $r \le e^{C \sqrt{\log N}}$ and without monochromatic $k$-patterns.
\end{lemma}

To see this, note that a standard extension of Behrend's construction creates a $k$-pattern-free subset $S\subset\Z/N\Z$ of size at least $Ne^{-C\sqrt{\log N}}$ (see, e.g., \cite[Theorem 2.3]{Ruz93}). Taking $2\log N e^{C\sqrt{\log N}}\leq e^{C'\sqrt{\log N}}$ random translates of $S$ covers $\Z/N\Z$ with positive probability, producing a coloring where each color class is $k$-pattern-free.

To apply \cref{prop:conv} we will need to construct a function $F$ where we have control over $\wt\Lambda_k(F)$. Recall that $\wt\Lambda_k(F)$ is defined to be the expected value of $F(x_1,y_1)\cdots F(x_k,y_k)$ over $k$-tuples of points in $\TT$ where the $x$-coordinates form a $k$-AP and the $y$-coordinates satisfy
\[
\text{``the $k$-binomial equation'':} \qquad 
\sum_{i=1}^k (-1)^i \binom{k-1}{i-1} y_i = 0.
\]

Following Ruzsa~\cite{Ruz93}, we say that a solution $(n_1,\ldots,n_k)$ to the equation $\sum_{i=1}^ka_in_i=0$ is \emph{trivial} if $\sum_{i:n_i=n}a_i=0$ for all $n$. For example, the trivial solutions to the 4-binomial equation $n_1-3n_2+3n_3-n_4=0$ are exactly the solutions satisfying $n_1=n_4$ and $n_2=n_3$.

\begin{lemma}[{Ruzsa \cite[Theorem 2.1]{Ruz93}}] \label{lem:avoid-nontriv}
For every sequence $a_1 ,\ldots, a_k$ of nonzero integers with $a_1 + \cdots + a_k = 0$, there is some constant $C = C(a_1, \ldots, a_k)$ so that whenever $N > Cr^{k-1}$, there is an $r$-element subset of $\ZZ/N\ZZ$ without non-trivial solutions to $a_1 n_1 + \cdots + a_k n_k = 0$.
\end{lemma}

This proof follows by simply constructing the desired set greedily.

\begin{definition} \label{def:beta_k}
For each even $k \ge 4$, let $\beta_k$ be the smallest real number such that the following is true: the largest subset of $[N]$ avoiding non-trivial solutions to the $k$-binomial equation has size at least $N^{1/\beta_k - o(1)}$, where $o(1) \to 0$ as $N \to \infty$.
\end{definition}

It follows from \cref{lem:avoid-nontriv} that $\beta_k\leq k-1$. Though we will not use this fact, it follows from a simple double-counting argument that for each even $k\geq 4$ we have $\beta_k\geq k/2$.
We conjecture that $\beta_k = k/2$ for all even $k\geq 4$, though we are only able to prove this in the case of $k=4$. This is the reason that \cref{thm:4-torus} is quantitatively stronger than \cref{thm:even-color-torus}.

\begin{remark}
$(n_1,n_2,\ldots,n_{k/2},n_{k/2},\ldots,n_1)$ is a trivial solution to the $k$-binomial equation for each even $k$ and each choice of $n_1,\ldots,n_{k/2}$. However, not all trivial solutions are of this form. For example, for $k=14$ the equality $\binom{13}3+\binom{13}7=\binom{13}4+\binom{13}8$ holds, so $(a,a,a,b,b,a,a,b,b,a,a,a,a,a)$ is also a trivial solution to the 14-binomial equation.
\end{remark}

\begin{lemma} \label{lem:13avoid}
For every integer $r$ and $N>Cr^2$, there exists a set $S\subseteq\Z/N\Z$ of size $r$, such that if $a, b, c, d\in S$ satisfies $a-3b+3c-d=0$, then $a=d$ and $b=c$. 
\end{lemma}

\begin{proof}
Let $S$ consist of the first $r$ positive integers whose base-$9$ expansion uses only $0,1,2$. Suppose $a + 3c = d + 3b$ for some $a,b,c,d \in S$. By comparing both sides in base-$3$, we find that $a=d$ and $b=c$. The largest element of $S$ is at most $9r^2$, so viewing $S$ as a subset of $\Z/N\Z$ for any $N>36r^2$ has the desired property.
\end{proof}

\section{Constructing a subset of the 2-torus}
\label{sec:odd-torus-const}

To prove \cref{thm:odd-torus,thm:even-color-torus,thm:4-torus} we will need to construct an appropriate function $F\colon\T^2\to[0,1]$. We start by constructing a coloring of $\Z/N\Z$ that avoids patterns corresponding to non-trivial solutions of the $k$-binomial equation.

If $n_1,\ldots, n_k$ is a trivial solution to the $k$-binomial equation, then one of the following holds:
\begin{enumerate}[(i)]
	\item $k$ is even and $n_i= n_{k+1-i}$ for all $i = 1, \dots, k/2$; or
	\item there is some $I \subset [k]$ with $\abs{I} \ge 3$ and $\sum_{i \in I} (-1)^i \binom{k-1}{i-1} = 0$ and $\{n_{j_i}\}_{i \in I}$ are all equal.
\end{enumerate}
(The converse might not hold since (ii) only considers one of the zero-sum subsets of coefficients.)
In light of this we make the following definition.

\begin{definition}[$k$-binomial pattern]
\label{defn:k-binomial}
	In a coloring $\phi \colon G \to [r]$ we say that $n,n+d,\ldots,n+(k-1)d$ with $d\neq 0$ is a \emph{$k$-binomial pattern} if either of the following holds:
\begin{enumerate}[(a)]
	\item $k$ is even and $\phi(n + id) = \phi(n + (k-1-i)d)$ for each $i = 1, \dots, k/2$;  or
	\item there is some $I \subset [k]$ with $\abs{I} \ge 3$ and $\sum_{i \in I} (-1)^i \binom{k-1}{i-1} = 0$ such that $\{n+(i-1)d\}_ {i \in I}$, are assigned the same color under $\phi$.
\end{enumerate}
\end{definition}

\begin{lemma}[General construction for $k$-APs] \label{lem:gen-k-ap-construction}
Let $k \ge 4$ be an integer. For any $r\geq 1$ and $\epsilon>0$, suppose there exist
\begin{itemize}
	\item a Jordan measurable $r$-coloring of $\TT$ such that \[\Pr_{x,y\in\TT}(x,x+y,\ldots,x+(k-1)y\text{ forms a $k$-binomial pattern})\leq\epsilon,\] and
	\item an $r$-element subset of $\ZZ/m\ZZ$ avoiding non-trivial solutions to the $k$-binomial equation.
\end{itemize}
Then there is some Jordan measurable $A \subset \TT^2$ such that $1_A$ has constant first marginal $1/(2^km)$ and $\wt \Lambda_k(1_A) \leq \epsilon m^{-k+1}$.
\end{lemma}

\begin{proof}
Let $\Phi \colon \TT \to [r]$ be an $r$-coloring with low $k$-binomial pattern density.
Let $S \subset \ZZ/m\ZZ$ be an $r$-element subset avoiding non-trivial solutions to the $k$-binomial equation and let $s_1, \ldots, s_{r}$ be the elements of $S$.

Define 
\[
J_j = \left[\frac{s_j}{m}, \frac{s_j}{m}  + \frac{1}{2^km}\right) \subset \TT, \quad \text{for each } j = 1, \ldots, r.
\]
Then set
\[
A = \{(x,y)\in\TT^2: y\in J_{\Phi(x)}\}.
\]
Each vertical slice of $A$ is an interval of length $1/(2^km)$ with vertical shift determined by the coloring $\Phi$ and the set $S$. Thus $A$ has constant first marginal $1/(2^km)$. 

It remains to show that $\wt \Lambda_k(1_A) \leq \epsilon m^{-k+1}$. 
Suppose $(x_1, y_1), \ldots, (x_k, y_k) \in A$ where $x_1, \dots, x_k$  is a $k$-AP in $\TT$ and $\sum_{i=1}^k (-1)^i \binom{k-1}{i-1} y_i = 0$ in $\TT$.
Then
\[
\frac{s_{\Phi(x_i)}}{m} \le y_i < \frac{s_{\Phi(x_i)}}{m} + \frac{1}{2^km}, \quad \text{ for each } i = 1, \ldots, k.
\]
Since $\sum_{i=1}^k (-1)^i \binom{k-1}{i-1} y_i = 0$ in $\TT$, this implies that 
\[
\sum_{i=1}^k (-1)^i \binom{k-1}{i-1}\frac{s_{\Phi(x_i)}}{m} 
\]
lies within less than $\frac{1}{2^k m} \sum_{i=1}^k \binom{k-1}{i-1} = \tfrac1{2m}$ of an integer. Thus
\begin{equation*}
\sum_{i=1}^k (-1)^i \binom{k-1}{i-1} s_{\Phi(x_i)} \equiv 0 \pmod m.
\end{equation*}
Since $S$ has no non-trivial solutions to the $k$-binomial equation, the above must be a trivial solution. 
By the definition of a trivial solution, one of the following must be true:
\begin{enumerate}[(i)]
	\item $k$ is even and $\Phi(x_i) = \Phi(x_{k+1-i})$ for all $i = 1, \dots, k/2$;
	\item there is some $I \subset [k]$ with $\abs{I} \ge 3$ and $\sum_{i \in I} (-1)^i \binom{k-1}{i-1} = 0$ and $\{x_i\}_{i \in I}$, are assigned the same color under $\Phi$.
\end{enumerate}
In other words $x_1,\ldots,x_k$ must form a $k$-binomial pattern in $\Phi$. By hypothesis, this occurs with probability at most $\epsilon$. Once the $x_i$'s are chosen, each $y_i$ lies in an interval of length $1/(2^km)$, subject to a single linear equation $\sum_{i=1}^k (-1)^i \binom{k-1}{i-1} y_i=0$. It follows that $\wt \Lambda_k(1_A) \leq \epsilon (2^km)^{-k+1}$, as desired.
\end{proof}

To apply \cref{lem:gen-k-ap-construction}, it will be convenient to have the following result that turns an $r$-coloring of $[N]$ without $k$-binomial patterns into a $k^2r$-coloring of $\TT$ with few $k$-binomial patterns.

\begin{lemma}
\label{lem:disc-to-cont-k-ap}
For $k\geq 4$, let $\phi$ be an $r$-coloring of $[N]$ that avoids monochromatic $k$-patterns and if $k$ is even also avoids nontrivial symmetrically colored $k$-APs. Then there exists a Jordan measurable $k^2r$-coloring $\Phi\colon \TT\to[k^2r]$ such that
\[\Pr_{x,y\in\TT}(x,x+y,\ldots,x+(k-1)y\text{ forms a $k$-binomial pattern in }\Phi)\lesssim_k \frac1N.\]
Furthermore, if $k=4$, the hypothesis that $\phi$ avoids monochromatic $k$-patterns can be dropped.
\end{lemma}

\begin{proof}
We cut $\TT$ into $k$ intervals and color each with $k$ interlaced copies of $\phi$; each of the $k^2$ copies of $\phi$ uses a disjoint set of colors. More precisely, partition $\TT$ into $k^2N$ equal-length intervals $\{I_{a,b,c}\}_{a,c\in[k];b\in[N]}$ where
\[
I_{a,b,c} = \left[ \frac{(a-1)kN+(b-1)k+c-1}{k^2N}, \frac{(a-1)kN+(b-1)k+c}{k^2N} \right) \subset \TT.
\]
Then define $\Phi\colon\TT\to[k^2r]$ by $\Phi(x)=((a-1)k+c-1)r+\phi(b)$ for $a,b,c$ such that $x\in I_{a,b,c}$. Suppose that $x_1,\ldots,x_k$ forms a $k$-binomial pattern in $\Phi$. Say $x_i\in I_{a_i,b_i,c_i}$. By definition, one of the following holds.
\begin{enumerate}[(i)]
	\item $k$ is even and $\Phi(x_i) = \Phi(x_{k+1-i})$ for all $i = 1, \dots, k/2$;
	\item there is some $I \subset [k]$ with $\abs{I} \ge 3$ and $\sum_{i \in I} (-1)^i \binom{k-1}{i-1} = 0$ and $\{x_i\}_{i \in I}$, are assigned the same color under $\Phi$.
\end{enumerate}

Let us analyze the two cases separately.  In both cases, we want to deduce that $(a_1,b_1,c_1)$, $\dots$, $(a_k,b_k,c_k)$ cannot all be distinct.

Suppose (i) holds. 
Since $\Phi(x_{i}) = \Phi(x_{k+1-i})$, we see that $a_i=a_{k+1-i}$ and $c_i=c_{k+1-i}$. Recall that \[x_i \in I_{a_i,b_i,c_i} = \frac{(a_i-1)N+(b_i-1)}{kN} +\left[\frac{c_i-1}{k^2N},\frac{c_i}{k^2N}\right).\] For $i=k/2$, this implies that the common difference $x_{k/2+1}-x_{k/2}=d/kN+\delta$ for some $d\in\{-(N-1),\ldots,N-1\}$ and $\delta\in(-1/k^2N,1/k^2N)$. Without loss of generality, suppose that $\delta\in[0,1/k^2N)$.

This implies that the smallest $i>k/2+1$ with $c_i\neq c_{k/2}$ (if it exists) satisfies $c_i\equiv c_{k/2}+1 \pmod k$. Similarly, the largest $i<k/2$ with $c_i \neq c_{k/2}$ satisfies  $c_i \equiv c_{k/2} - 1 \pmod k$. However, since we assumed that $c_i= c_{k+1-i}$ for all $i$, this implies that actually $c_1=c_2=\cdots=c_k$.

Since $|x_{k/2+1}-x_{k/2}|<1/k$, the same argument shows that $a_1=a_2=\cdots=a_k$. Therefore $b_1,\ldots, b_k$ form a symmetrically colored $k$-AP in $\phi$. Using the fact that $\phi$ has no nontrivial symmetrically colored $k$-APs, we find that $b_1 = \cdots = b_k$.

Now suppose (ii) holds.
Since $\Phi(x_i)$ are equal for all $i\in I$, we must have some $a,c\in[k]$ such that $a_i=a$ and $c_i=c$ for all $i\in I$. Now we wish to show that three of $b_1,\ldots,b_k$ form a monochromatic $k$-pattern in $\phi$.

Take $i_1<i_2<i_3\in I$. Since $x_1,\ldots,x_k$ form a $k$-AP in $\TT$ we know that
\[(i_3-i_2)(x_{i_2}-x_{i_1})\equiv (i_2-{i_1})(x_{i_3}-x_{i_2})\pmod 1.\]
Using the fact that \[x_i\in I_{a,b_i,c}=\frac{(a-1)kN+c}{k^2N}+\left[\frac{(b_i-1)k-1}{k^2N},\frac{(b_i-1)k}{k^2N}\right)\] for all $i\in I$, we conclude that
\[\left[(i_3-i_2)\frac{b_{i_2}-b_{i_1}}{kN}+\left(\frac{i_2-i_3}{k^2N},\frac{i_3-i_2}{k^2N}\right)\right]\cap \left[(i_2-i_1)\frac{b_{i_3}-b_{i_2}}{kN}+\left(\frac{i_1-i_2}{k^2N},\frac{i_2-i_1}{k^2N}\right)\right]\neq\emptyset.\]
Since $i_3-i_1<k$, this further implies that
\[(i_3-i_2)(b_{i_2}-b_{i_1})\equiv(i_2-i_1)(b_{i_3}-b_{i_2})\pmod{kN}.\]
Finally, since $|b_{i_2}-b_{i_1}|,|b_{i_3}-b_{i_2}|<N$ and $|i_3-i_2|+|i_2-i_1|<k$, this gives
\[(i_3-i_2)(b_{i_2}-b_{i_1})=(i_2-i_1)(b_{i_3}-b_{i_2}).\]
This implies that either $b_{i_1},b_{i_2},b_{i_3}$ are a monochromatic $k$-pattern in $\phi$ or are all equal. Since the former cannot occur, we conclude that $b_{i_1}=b_{i_2}=b_{i_3}$.

In either case we have shown that if $x_1,\ldots,x_k$ forms a $k$-binomial pattern in $\Phi$ then there exist at least two distinct indices $i\neq i'\in [k]$ such that $(a_i,b_i,c_i)=(a_{i'},b_{i'},c_{i'})$, i.e., such that $x_i,x_{i'}$ lie in the same interval of length $1/k^2N$. There are at most $k^2$ choices for $i,i'$, and then $x_{i'}$ lies in the same interval as $x_i$ with probability $1/k^2N$. Finally, once $x_i,x_{i'}$ are fixed, there are at most $k$ choices for the whole sequence $x_1,\ldots, x_k$. Therefore we see that the $k$-binomial pattern density in $\Phi$ is at most $k/N$, as desired.

Finally if $k=4$, note that if (ii) holds, then $I=\{1,2,3,4\}$ since it corresponds to a subset of $\{-1,3,-3,1\}$ which sums to 0. Thus case (ii) when $k=4$ corresponds to a monochromatic 4-AP in $\Phi$ which was already ruled out by case (i). This completes the proof.
\end{proof}

\cref{thm:odd-torus} for odd length APs then follows.

\begin{proof}[Proof of \cref{thm:odd-torus}]
We will apply \cref{lem:gen-k-ap-construction} to prove the claim. 

By \cref{lem:avoid-pattern} there is a coloring $\phi\colon [N] \to [r]$ that avoids monochromatic $k$-patterns where $N=r^{\Theta_k(\log r)}$. 
Since $k$ is odd, $\phi$ satisfies the hypothesis of \cref{lem:disc-to-cont-k-ap} and thus there is a $k^2r$-coloring $\Phi\colon\TT\to[k^2r]$ with $k$-binomial pattern density $\lesssim_k 1/N$.

Next, by \cref{lem:avoid-nontriv}, there exists a $k^2r$-element subset $S$ of $\Z/m\Z$ avoiding non-trivial solutions to the $k$-binomial equation, where $m=\Theta_k(r^{k-1})$. 
 
Applying \cref{lem:gen-k-ap-construction} to $\Phi$ and $S$ yields $1_A\subset \TT^2$ with constant first marginal $1/(2^km)$ and $\wt \Lambda_k(1_A) \lesssim_k m^{-\Theta_k(\log m)}$. By definition $\wt \rho_k(1/(2^km)) \lesssim_k m^{-\Theta_k(\log m)}$. By the monotonicity of $\wt\rho_k$ this implies the desired bound on $\wt\rho_k(\alpha)$.
\end{proof}

\section{Even length APs: from colorings to uniform sets}
\label{sec:even-torus-const}

\begin{lemma}[Tensor power construction] \label{lem:power}
Suppose there exists an $r$-coloring of $\ZZ/N\ZZ$ avoiding symmetrically colored $k$-APs.
Then for every positive integer $\ell$, there exists an $r^\ell$-coloring of $\ZZ/N^\ell\ZZ$ avoiding symmetrically colored $k$-APs.	
\end{lemma}

\begin{proof}
The idea is to write an element of $\ZZ/N^\ell \ZZ$ in its base $N$ expansion, and then color each digit separately. More precisely, given $\phi \colon \ZZ/N\ZZ \to [r]$ avoiding symmetrically colored $k$-APs, define the coloring $\psi \colon \ZZ/N^\ell \ZZ \to [r]^\ell$ as follows. Every element of $\ZZ/N^\ell\ZZ$ can be written uniquely as $n = n_0 + n_1 N + \cdots + n_{\ell -1}N^{\ell - 1}$ with $n_0, \ldots, n_{\ell -1} \in \set{0, 1, \ldots, N-1}$. Set $\psi(x) = (\phi(n_0), \ldots, \phi(n_{\ell -1}))$. 
Then $\psi$ has no symmetrically colored $k$-APs. 
Indeed, suppose $a_1,\ldots,a_k \in \ZZ/N^\ell \ZZ$ is a non-trivial $k$-AP.
Let $i$ to be the index of the least significant base-$N$ digit where $a_1,\ldots,a_k$ are not all equal. 
By considering mod $N^{i+1}$, the $i$-th digits $a_{1,i},\ldots,a_{k,i}$ form a non-trivial $k$-AP in $\ZZ/N\ZZ$, and thus they are not symmetrically colored under $\phi$. Hence $a_1,\ldots,a_k$ is not symmetrically colored under $\psi$.
\end{proof}

\begin{proof}[Proof of \cref{thm:4-torus}]
Suppose there exists an $r$-coloring of $\ZZ/N\ZZ$ avoiding symmetrically colored 4-APs.
Let $0 < \alpha < 1/2$.
By \cref{lem:power}, for every positive integer $\ell$, there is an $r^\ell$-coloring of $\ZZ/N^\ell\ZZ$ that avoids symmetrically colored 4-APs.
By \cref{lem:disc-to-cont-k-ap}, there is a $16r^\ell$-coloring of $\TT$ whose density of 4-binomial patterns (i.e., symmetrically colored 4-APs) is $\lesssim N^{-\ell}$.
By \cref{lem:13avoid}, for some $m = O(r^{2\ell})$, there is a $16r^{\ell}$-element subset of $\Z/m\Z$ that avoids solutions to the $4$-binomial equation.
By \cref{lem:gen-k-ap-construction},  we obtain $A \subset \TT^2$ with constant first marginal $1/(16m)\gtrsim r^{-2\ell}$ and $\wt \Lambda_4(1_A) \lesssim N^{-\ell}r^{-6\ell}$. 
Thus we can choose $\ell =  (1/2)\log_r(1/\alpha) - O_{r}(1)$ so that the marginals are $1/(16m) \ge \alpha$, and we have
\[
\wt \Lambda_4(1_A) \lesssim N^{-\ell}r^{-6\ell} \lesssim_{N,r} \alpha^{3+(1/2)\log_r N}.
\]
This completes the proof by the monotonicity of $\wt\rho_4$.
\end{proof}

\begin{proof}[Proof of \cref{thm:even-color-torus}]
We will apply \cref{lem:gen-k-ap-construction} to prove the claim. 

First, we show that there exists a coloring $\phi\colon [N^\ell] \to [u]$ that avoids monchromatic $k$-patterns and symmetrically colored $k$-APs where $u=r^\ell e^{O(\sqrt {\ell \log N})}$. Let $\psi\colon  \Z/N\Z \to [r]$ be the $r$-coloring of $\Z/N\Z$ avoiding symmetrically colored $k$-APs. By \cref{lem:power}, for any $\ell \ge 1$, there exists an $r^\ell$-coloring $\psi_\ell$ of $\Z/N^\ell\Z$ avoiding symmetrically colored $k$-APs.
We also apply \cref{lem:avoid-pattern} to obtain $\chi \colon \Z/N^\ell\Z\to [e^{O(\sqrt {\ell \log N})}]$ avoiding monochromatic $k$-patterns. Defining $\phi$ to be the product coloring of $\psi_\ell$ and $\chi$ gives the desired properties.

Then by \cref{lem:disc-to-cont-k-ap}, there is a $k^2u$-coloring $\Phi$ of $\TT$ with $k$-binomial pattern density $\lesssim_k N^{-\ell}$.

Next, by definition of $\beta_k$, there exists a $k^2u$-element subset $S$ of $\Z/m\Z$ avoiding non-trivial solutions to the $k$-binomial equation, where $m=(k^2u)^{1/\beta_k+o_{\ell\to \infty}(1)}$. 
 
Applying \cref{lem:gen-k-ap-construction} to $\Phi$ and $S$ yields $1_A\subset \TT^2$ with constant first marginal \[1/(2^km)\gtrsim_k r^{-\ell/\beta_k+o_{\ell\to\infty}(1)}e^{-O(\sqrt{\ell\log N})}\] and \[\wt \Lambda_k(1_A) \lesssim_k m^{-k+1}N^{-\ell}\lesssim_k u^{-(k-1)/\beta_k+o_{\ell\to\infty}(1)}N^{-\ell}=r^{-\ell(k-1)/\beta_k+o_{\ell\to\infty}(1)}N^{-(1+o_{\ell\to\infty}(1))\ell}.\]

Taking $\ell=\beta_k\log_r(1/\alpha)-O_{k,r,N}(1)$ gives $1/(2^km)\geq\alpha$ for all sufficiently small $\alpha>0$ and
\[\tilde\Lambda_k(1_A)\lesssim_{N,k,r}\alpha^{(k-1+\log_rN)/\beta_k-o_{\alpha\to0}(1)}.\]
Combining this with the monotonicity of $\wt\rho_k$ and the bound $\beta_k\leq k-1$ gives the desired result.
\end{proof}

\section{Extracting a coloring}\label{sec:extract}

In this section we prove \cref{thm:even-k-converse}, the converse to our main theorem. We do this via a randomized procedure that extracts a colorings from a function. More precisely, suppose that we are given a function $F \colon \TT^2 \to [0,1]$ with constant first marginal $\alpha$ and small $\wt \Lambda_k(F)$. For an appropriate choice of $r,N$ we construct an $r$-coloring of $\Z/N\Z$ that avoids symmetrically colored $k$-APs.

The main idea is as follows. We select $y_1,y_2,\ldots,y_r\in\TT$ at random. We convert $F$ to a partial coloring of $\TT$ by giving $x$ the color $i$ if $i$ is the smallest integer such that $F(x,y_i)$ is at least some threshold. We then convert this partial coloring of $\TT$ into a partial coloring of $[N]$ by randomly embedding an $N$-AP into $\TT$. We will prove that, assuming $r,N$ are chosen appropriately, with positive probability this partial coloring is actually a full coloring and avoids symmetrically colored $k$-APs.

\cref{thm:even-k-converse} follows immediately from the next lemma.

\begin{lemma}
For $k \ge 4$ even there exists $c_k>0$ such that the following holds. Suppose $F \colon \TT^2 \to [0,1]$ has constant first marginal $\alpha \in (0,1)$ and $\wt\Lambda_k (F) = \eta$. 
Then whenever
\[
r \ge (\alpha k)^{-1}\log(1/\eta)
\quad\text{and}\quad
N \le c_k \eta^{-\frac{1}{2k}} \alpha^{\frac{3k}{4}} (\log(1/\eta))^{-\frac{k}{4}},
\]
there is an $r$-coloring of $[N]$ that avoids symmetrically colored $k$-APs.
\end{lemma}

\begin{proof}
Pick $x_0, x_1, y_1, \dots, y_r \in \TT$ independently and uniformly at random. We define $\phi \colon [N] \to [r]$ by setting $\phi(i)$ to be the least $j$ such that $F(x_0 + ix_1, y_j) \ge \alpha/2$, and leaving $\phi(i)$ undefined if no such $j$ exists. By Markov's inequality, $\PP_{x_0,x_1,y_j}(F(x_0+ix_1,y_j)<\alpha/2)\leq 1-\alpha/2$ for each choice of $i,j$. Since the $y_j$ are independent, by taking a union bound of the choices of $i$,
\[
\PP(\text{$\phi(i)$ is undefined for some $i$})
\le N \paren{1 - \alpha/2}^r \le N e^{-r \alpha/2} < \frac12,
\]
provided that $c_k$ is sufficiently small.

Next, let us upper bound the probability that there exists a symmetrically colored $k$-AP. 
We have
\begin{align*}
\eta = \wt \Lambda_k(F) 
&= \EE_{\substack{w_0,w_1, z_1, \dots, z_k: \\ \sum_{i=1}^j(-1)^i \binom{k-1}{i-1} z_i}= 0} \prod_{i=1}^k F(w_0 + iw_1, z_i)
\\
&\ge
\EE_{\substack{w_0,w_1, z_1, \dots, z_k: \\ \sum_{i=1}^j(-1)^i \binom{k-1}{i-1} z_i}= 0} \prod_{i=1}^k F(w_0 + iw_1, z_i)F(w_0 + (k+1-i) w_1, z_i)
\\
&=
\EE_{w_0, w_1, z} \paren{\EE_{\substack{z_1, \dots, z_{k/2}: \\ \sum_{i=1}^{k/2} (-1)^i \binom{k-1}{i-1} z_i= z}} \prod_{i=1}^{k/2} F(w_0 + iw_1, z_i)F(w_0 + (k+1-i) w_1, z_i)}^2
\\
&\ge 
\paren{\EE_{w_0, w_1, z} \EE_{\substack{z_1, \dots, z_{k/2}: \\ \sum_{i=1}^{k/2} (-1)^i \binom{k-1}{i-1} z_i= z}} \prod_{i=1}^{k/2} F(w_0 + iw_1, z_i)F(w_0 + (k+1-i) w_1, z_i)}^2
\\
&=
\paren{\EE_{w_0, w_1, z_1, \dots, z_{k/2}} \prod_{i=1}^{k/2} F(w_0 + iw_1, z_i)F(w_0 + (k+1-i) w_1, z_i)}^2.
\end{align*}
In the second line we used the trivial bound $F\leq 1$ to insert the $F(w_0+(k+1-i)w_1,z_i)$ terms. 

Given a non-trivial $k$-AP $a_1, \dots, a_k \in [N]$ and a map $\psi \colon [k] \to [r]$ that is symmetric in the sense of $\psi(i) = \psi(k+1-i)$ for all $i \le k/2$,  we have
\begin{align*}
\MoveEqLeft \PP_\phi (\phi(a_i) = \psi(i) \text{ for all } i\in[k])
\\
&\le 
\PP_{x_0, x_1, y_1, \dots, y_r} 
\paren{ F(x_0 + a_i x_1, y_{\psi(i)}) \ge \frac{\alpha}{2}\text{ for all } i\in[k]}
\\
&\le 
\paren{\frac{2}{\alpha}}^k \EE_{x_0, x_1, y_1, \dots, y_r}  \prod_{i=1}^k F(x_0 + a_i x_1, y_{\psi(i)})
\\
&\le 
\paren{\frac{2}{\alpha}}^k \paren{\EE_{x_0, x_1} \paren{\EE_{y_1, \dots, y_r}  \prod_{i=1}^k F(x_0 + a_i x_1, y_{\psi(i)})}^{k/2}}^{2/k}
\\
&=
\paren{\frac{2}{\alpha}}^k \paren{\EE_{x_0, x_1} \paren{\EE_{y_1, \dots, y_r}  \prod_{i=1}^{k/2} F(x_0 + a_i x_1, y_{\psi(i)})F(x_0 + a_{k+1-i} x_1, y_{\psi(i)})}^{k/2}}^{2/k}
\\
&\le 
\paren{\frac{2}{\alpha}}^k
\paren{\EE_{x_0, x_1, z_1, \dots, z_{k/2}}
\prod_{i=1}^{k/2} 
F(x_0 + a_i x_1, z_i)
F(x_0 + a_{k+1-i} x_1, z_{k+1-i})
}^{2/k}
\\
&\le \paren{\frac{2}{\alpha}}^k \eta^{1/k}.
\end{align*}
Here the third line follows by Markov's inequality and the fourth by H\"older's inequality. To see the sixth line we expand the $k/2$-th power and again use the trivial bound of $F\leq 1$ to drop all but $k/2$ of the terms in each product. The last line follows from the previous calculation after the change of variables $w_0=x_0+ax_1$ and $w_1=dx_1$ where the $k$-AP $a_1,\ldots,a_k$ is given by $a_i=a+id$.

There are at most $N^2$ choices for the $k$-AP $a_1, \dots, a_k \in [N]$ and at most $r^{k/2}$ choices for $\psi$, so by a union bound we have
\[
\PP(\phi \text{ contains a symmetrically colored $k$-AP})
\le N^2 r^{k/2} \paren{\frac{2}{\alpha}}^k \eta^{1/k} < \frac12.
\]
The last inequality follows for $r=\lceil(\alpha k)^{-1}\log(1/\eta)\rceil$ and $c_k$ sufficiently small. It follows that with positive probability, $\phi$ is an $r$-coloring of $[N]$ avoiding symmetrically colored $k$-APs.
\end{proof}

\section{General patterns}

In this section, we prove \cref{thm:general_pattern} and \cref{thm:general_reduction}. While the main proof strategy is the same as before, there are some additional technical challenges that we need to address.

\begin{definition}
Let $\bfa = (a_1, \dots, a_k)$ have increasing integer coordinates. The \emph{$\bfa$-coefficients} $c(\bfa)=(c_1, \dots, c_k)$ are defined as $c_i = \prod_{j\ne i} (a_i - a_j)$ for $1\le i\le k$. Define the \emph{$\bfa$-binomial equation} in variables $n_1,\ldots,n_k$ to be
    \[\sum_{i=1}^kn_i/c_i=0.\]
\end{definition} 

Note that for $\bfa=(0,1,\ldots,k-1)$ we have $c_i=(-1)^{k-i}(i-1)!(k-i)!$ so in this case the $\bfa$-binomial equation agrees with the $k$-binomial equation (up to a constant multiplicative factor of $(-1)^k(k-1)!$). The reason the $k$-binomial equation appears for $k$-APs is because the $(k-2)$-nd power of a $k$-AP satisfies the $k$-binomial equation. More precisely we have the equality of polynomials
\begin{equation}
\label{eq:k-binom-equality}
\sum_{i=1}^k(-1)^i\binom{k-1}{i-1} (x+(i-1)y)^{k-2}=0.
\end{equation}

We first show that this property extends to the $\bfa$-binomial equation for all $\bfa$.

\begin{lemma}
\label{lem:a-binom-eq}
Let $\bfa = (a_1, \dots, a_n)$ have increasing integer coordinates. Say $c(\bfa)=(c_1, \dots, c_n)$ are the corresponding $\bfa$-coefficients. Then we have the equality of polynomials $\sum_{i=1}^n (x+a_iy)^{n-2}/c_i=0$.
\end{lemma}

\begin{proof}
We first show that $\sum_{i=1}^n a_i^t/c_i=0$ holds for any $0\le t\le n-2$. Since the $a_i$ are distinct, it suffices to show that $\prod_{1\leq j<k\leq n}(a_j-a_k)\sum_{i=1}^n a_i^t/c_i=0$. Expanding, this reduces to showing that
\begin{equation*}
f_t\coloneqq\sum_{i=1}^n (-1)^ia_i^t\prod_{\genfrac{}{}{0pt}{}{1\leq j<k\leq n}{j,k\neq i}}(a_j-a_k)=0\qquad\text{for all }0\leq t\leq n-2.
\end{equation*}
It is easy to check that $f_t(a_1,\ldots,a_\ell,a_{\ell+1},\ldots,a_n)=-f_t(a_1,\ldots,a_{\ell+1},a_\ell,\ldots,a_n)$ for all $\ell$. This implies that $f_t$ is an alternating polynomial. Since every alternating polynomial is a multiple of the Vandermonde polynomial $V(a_1,\ldots,a_n)=\prod_{1\leq i<j\leq n}(a_i-a_j)$ and $\deg f_t=\binom{n-1}2+t<\binom n2=\deg V$, we conclude that $f_t= 0$ for $t\leq n-2$.

Therefore $\sum_{i=1}^n a_i^t/c_i=0$ holds for each $0\le t\le n-2$, so every coefficient of the polynomial $\sum_{i=1}^n (x+a_iy)^{n-2}/c_i$ is 0, as desired.
\end{proof}

We then have the following analogue of \cref{prop:conv}.

\begin{proposition} \label{prop:conv-a}
Given a Riemann integrable $F \colon \TT^2 \to [0,1]$,
define $f_N \colon \ZZ/N\ZZ \to [0,1]$ by $f_N(n) = F(n/N,n^{k-2}/N)$. 
If $F$ has constant first marginal $\alpha$, then, as $N \to \infty$, 
\[
\norm{f_N - \alpha}_{U^{k-2}} \to 0
\]
and
\[
\Lambda_\bfa(f_N) \to \wt\Lambda_\bfa(F),
\]
where
\[
\wt\Lambda_\bfa(F) \coloneqq \EE_{(x_1, \dots, x_k, y_1, \dots, y_k) \in V} F(x_1, y_1) \cdots F(x_k, y_k)
\]
where $V$ is the subset of $\TT^{2k}$ defined by all points $(x_1, \dots, x_k, y_1, \dots, y_k)$ where $x_1,\ldots,x_k$ is an $\bfa$-AP and $y_1,\ldots,y_k$ satisfies the $\bfa$-binomial equation.
\end{proposition}

This proof is identical to that of \cref{prop:conv} other than an application of \cref{lem:a-binom-eq} instead of \cref{eq:k-binom-equality}.

We now want to adapt \cref{defn:k-binomial} and 
\cref{lem:gen-k-ap-construction} to this setting. To do so we will need to understand trivial solutions to the $\bfa$-binomial equation.

\begin{definition}
Let $\bfa = (a_1, \dots, a_k)$ have increasing integer coordinates and let $c(\bfa)=(c_1, \dots, c_k)$ be the corresponding $\bfa$-coefficients. When $k$ is even, a \emph{pairing} of $\bfa$ is a map $f\colon[k]\to[k]$ which partitions $[k]$ into pairs (i.e., $f(f(i))=i$ for all $i$ and $f(i)\neq i$) where $c_i=-c_{f(i)}$ for all $i\in[k]$.
\end{definition}

Now it follows from the definition that if $n_1,\ldots,n_k$ is a trivial solution to the $\bfa$-binomial equation, then at least one of the following holds:
\begin{enumerate}[(i)]
    \item $k$ is even and there exists a pairing $f$ of $\bfa$ such that $n_i=n_{f(i)}$ for all $i\in[k]$; or
    \item there is some $I\subset [k]$ with $|I|\geq 3$ and $\sum_{i\in I}c_i^{-1}=0$ and $\{n_i\}_{i\in I}$ are all equal to each other.
\end{enumerate}

In light of this observation we make the following definition.

\begin{definition}[$\bfa$-binomial pattern]
Let $\bfa = (a_1, \dots, a_k)$ have increasing integer coordinates and let $c(\bfa)=(c_1, \dots, c_k)$ be the corresponding $\bfa$-coefficients. 
For a coloring $\phi \colon G \to [r]$, we say $n+a_1d, n+a_2d, \dots, n+a_kd$ with $d\neq 0$ is an \emph{$\bfa$-binomial pattern} if either of the following holds:
\begin{enumerate}[(a)]
	\item $k$ is even and there is some pairing $f$ of $\bfa$  such that $\phi(n + a_id) = \phi(n + a_{f(i)}d)$ for each $i\in[k]$; or
	\item there is some $I \subset [k]$ with $\abs{I} \ge 3$ and $\sum_{i \in I} c_i^{-1} = 0$ such that $\{n+a_id\}_{i \in I}$, are assigned the same color under $\phi$.
\end{enumerate}
\end{definition}

\begin{remark}
\label{rem:asym-pairing}
If $\bfa$ is symmetric (recall that this means that $k$ is even and $a_1+a_k=a_2+a_{k-1}=a_3+a_{k-2}=\cdots$) then clearly $f(i)=k-i+1$ is a pairing. (We call this pairing the \emph{symmetric pairing}.) However this is not the only possibility. For example, $\bfa=(1,2,10,16,17,20)$ is not symmetric, but $c(\bfa)=(-41040, 30240, -30240, 5040, -5040, 41040)$, so $\bfa$ has a pairing. Furthermore, consider the symmetric $\bfa=(1,2,3,6,7,8)$. Here $c(\bfa)=(-420, 120, -120, 120, -120, 420)$, showing that symmetric $\bfa$ can have pairings other than the symmetric one.
\end{remark}

\begin{lemma}[General construction for $\bfa$-AP] \label{lem:gen-a-ap-construction}
Let $\bfa = (a_1, \dots, a_k)$ have increasing integer coordinates. There is a constant $C_\bfa$ depending on $\bfa$ such that the following holds. For any $r\geq 1$ and $\epsilon>0$, suppose there exist
\begin{itemize}
    % \item an $r$-coloring of $\Z/N\Z$ that avoids $\bfa$-binomial patterns, and
	\item a Jordan measurable $r$-coloring $\Phi\colon\TT\to[r]$ with
    \[\Pr_{x,y\in\TT}(x+a_1y,\ldots,x+a_ky\text{ forms an $\bfa$-binomial pattern})\leq\epsilon,\] and
	\item an $r$-element subset of $\ZZ/m\ZZ$ avoiding non-trivial solutions to the $\bfa$-binomial equation.
\end{itemize}
Then there is some Jordan measurable $A \subset \TT^2$ such that $1_A$ has constant first marginal $1/(C_\bfa m)$ and $\wt \Lambda_{\bfa}(1_A) \leq\epsilon m^{-k+1}$.
\end{lemma}

This lemma follows from the same proof as \cref{lem:gen-k-ap-construction}.

We will use \cref{lem:gen-a-ap-construction} to prove \cref{thm:general_reduction,thm:general_pattern}. To do this we need a coloring that avoids $\bfa$-binomial patterns and a set that avoids non-trivial solutions to the $\bfa$-binomial equation. The latter is produced by \cref{lem:avoid-nontriv}. The former is harder to produce for $\bfa$-APs than for $k$-APs. This is due to the presence of non-symmetric pairings.
In the next lemmas we show how to produce efficient colorings that avoid all non-symmetric pairings first in $\Z/N\Z$ and then in $\TT$.

To provide the coloring, we first define a simple coloring of $\Z/m\Z$ that helps to handle wrap-around issues. 

\begin{definition}
For an abelian group $G$ and $p\in G$,  a \emph{$k$-AP with jumps of size $p$} is a sequence $n_1, \ldots, n_k\in G$ such that there exists $d\in G$ (the common difference) such that $n_{i+1}-n_i \in \{d,d+p\}$ for all $1\le i<k$. 
\end{definition}

The following lemma is useful for showing that certain APs with jumps are genuine APs.

\begin{lemma} \label{lem:error-correct}
Let $a,k,p$ be positive integers, with $k\leq a$ and $p$ relatively prime to $a!$. Let $G=\Z$ or $G=\Z/m\Z$ for some $m$ a multiple of $a!$. Suppose that $n_1, \ldots, n_k\in G$ is a $k$-AP with jumps of size $p$. If $n_1\equiv n_k \pmod{a!}$, then $n_1, \ldots, n_k$ is a $k$-AP. Furthermore if there exist $1<k'<k''<k$, such that $n_1\equiv n_{k''} \pmod{a!}$ and $n_{k'}\equiv n_k\pmod{a!}$, then $n_1, \ldots, n_k$ is a $k$-AP. 
\end{lemma}

\begin{proof}
Let $d\in G$ be the common difference. Let $0\leq c<k$ be the number of $1\leq i<k$ such that $n_{i+1}-n_i= d+p$. Then $n_k-n_1=d(k-1)+cp$.

For the first statement, we have that $d(k-1)+cp=n_k-n_1\equiv 0 \pmod{a!}$. Since $k-1$ is a divisor of $a!$ we have $(k-1)\mid (n_k-n_1)$ implying that $(k-1)\mid cp$. As $p$ is relatively prime to $k-1$, we conclude that either $c=k-1$ or $c=0$, implying that $n_1, \ldots, n_k$ is an arithmetic progression with common difference $d+p$ or $d$. This proves the first result.

For the second statement, we have by the first statement that $n_1,\ldots, n_{k''}$ is an AP and $n_{k'},\ldots, n_k$ is an AP. Since $k'<k''$, this implies that $n_1,\ldots, n_k$ is a $k$-AP, as desired.
\end{proof}

We will construct the desired coloring using an intricate Behrend-style coloring that avoids the following types of patterns.

\begin{definition}
Let $\bfa=(a_1,a_2,a_3,a_4)$ have increasing integer coordinates. In a coloring $\phi$, an \emph{$\bfa$-ABAB pattern} is an $\bfa$-AP $(n+a_1d,n+a_2d,n+a_3d,n+a_4d)$ such that $\phi(n+a_1 d)=\phi(n+a_3 d)$ and $\phi(n+a_2 d)=\phi(n+a_4d)$. Similarly, an \emph{asymmetric $\bfa$-ABBA pattern} is an $\bfa$-AP such that $\phi(n+a_1 d)=\phi(n+a_4 d)$ and $\phi(n+a_2 d)=\phi(n+a_3d)$ and such that $\bfa$ satisfies $a_1+a_4\neq a_2+a_3$.
\end{definition}

\begin{lemma}
\label{lem:mod-behrend-coloring}
For all $a$, there is a constant $C$ such that the following holds. For every $r\ge 2$, there exists some $N=\Theta(r^{C\log r})$ such that there is an $r$-coloring of $\Z/N\Z$ that avoids $\bfa$-ABAB patterns and asymmetric $\bfa$-ABBA patterns for every $\bfa=(a_1,a_2,a_3,a_4)$ with $0<a_1<a_2<a_3<a_4\leq a$.
\end{lemma}

\begin{proof}
We will pick some appropriate integers $M,m$ where $M$ is relatively prime to $a!$. First consider the coloring $\Psi\colon\{0,1,\ldots,M-1\}^m\to [mM^2]$ defined by $\Psi(n_1,\ldots,n_m)=n_1^2+\cdots+n_m^2+1$. First we claim that $\Psi$ avoids $\bfa$-ABAB patterns and asymmetric $\bfa$-ABBA patterns, for all $0<a_1<a_2<a_3<a_4\leq a$. To see this, note that for a given $n, d\in \Z^m$, on the line $n+id$, the coloring function $\Psi(n+id)=f_{n,d}(i)$ is a non-trivial quadratic function of $i$. If $(n+a_1d,n+a_2d,n+a_3d,n+a_4d)$ is an $\bfa$-ABAB pattern then $f_{n,d}(a_1)=f_{n,d}(a_3)$, so the axis of symmetry of the quadratic $f_{n,d}$ is at $(a_1+a_3)/2$. However, as $f_{n,d}(a_2)=f_{n,d}(a_4)$, the axis of symmetric would have to be at $(a_2+a_4)/2>(a_1+a_3)/2$, contradiction. Similarly, an asymmetric $\bfa$-ABBA pattern is not possible since the axis of symmetry would have to be at $(a_1+a_4)/2\neq(a_2+a_3)/2$.

Write $N=M^m$ and identify $\Z/N\Z$ with $\{0,1,\ldots,M-1\}^m$ by considering the base $M$ expansion of an element $0\leq n<N$. Thus we can view $\Psi$ as a coloring $\Psi\colon\Z/N\Z\to[mM^2]$.

We will define another coloring $\chi \colon \{0,1,\ldots,M-1\}^m\to\{0,1,\ldots,a!-1\}^m$ that reduces the coordinates $\bmod{\,a!}$. Namely define $\chi(n_1,\ldots,n_m)=(n_1\mod {a!},\ldots,n_m\mod {a!})$. Finally define $\psi\colon\Z/N\Z\to[mM^2]\times\{0,1,\ldots,a!-1\}^m$ to be the product of $\Psi$ and $\chi$. We show that $\psi$ has the desired properties.

To prove this we claim the following. Let $(n+a_1d,n+a_2d,n+a_3d,n+a_4d)$ be an $\bfa$-AP in $\Z/N\Z$. If it is an $\bfa$-ABAB pattern or an $\bfa$-ABBA pattern with respect to $\chi$, then it is also an $\bfa$-AP when viewed as a subset of $\{0,1,\ldots,M-1\}^m$.

This claim follows by verification digit by digit. Suppose that $(n+a_1d,n+a_2d,n+a_3d,n+a_4d)$ is an $\bfa$-ABAB pattern or an $\bfa$-ABBA pattern with respect to $\chi$. Let $0\leq m_i<M$ be the least significant digit in the base $M$ expansion of $n+id$. We see that $m_{a_1}, m_{a_1+1}, \ldots, m_{a_4}$ is an AP with jumps of size $M$ in $\Z$.

Suppose that $(n+a_1d,n+a_2d,n+a_3d,n+a_4d)$ is an $\bfa$-ABBA pattern. By the definition of $\chi$, we know that $m_{a_1}\equiv m_{a_4}\pmod{a!}$. By \cref{lem:error-correct}, $m_{a_1}, m_{a_1+1}, \ldots, m_{a_4}$ is a genuine AP. Hence $m_{a_1},m_{a_2},m_{a_3},m_{a_4}$ form an $\bfa$-AP viewed as elements of $\{0,1,\ldots,M-1\}$. Then we can ignore the last digit and induct, since $((n+a_id)-m_{a_i})/M$ for $1\le i\le 4$ is an $\bfa$-ABBA pattern in $\{0,1,\ldots,N/M-1\}$ with respect to $\chi$. 

Similarly, suppose $(n+a_1d,n+a_2d,n+a_3d,n+a_4d)$ is an $\bfa$-ABAB pattern. By definition of $\chi$, $m_{a_1}\equiv m_{a_3}\pmod{a!}$ and $m_{a_2}\equiv m_{a_4}\pmod{a!}$ so by \cref{lem:error-correct}, $m_{a_1}, m_{a_1+1}, \ldots, m_{a_4}$ is an AP. In conclusion, $m_{a_1}, m_{a_2},m_{a_3},m_{a_4}$ still form an $\bfa$-AP viewed as elements of $\{0,1,\ldots,M-1\}$ so we can induct as in the previous case.

This completes the proof of the claim.

Thus we have showed that for any $\bfa=(a_1,a_2,a_3,a_4)$ with $a_1<a_2<a_3<a_4\leq a$ any $\bfa$-ABAB pattern or asymmetric $\bfa$-ABBA pattern in the coloring $\psi$ of $\Z/N\Z$ is actually such a pattern when viewed as a subset of $\{0,\ldots,M-1\}^m$. However we chose $\Psi$ in a way that no such pattern exists.

Picking $M=\Theta(r^c)$ relatively prime to $a!$ and $m=c'\log r$ for appropriate $c,c'$ gives an $r$-coloring of $\Z/N\Z$ with $N\geq r^{\Omega(\log r)}$, as desired.
\end{proof}

The reason we study ABAB patterns and asymmetric ABBA patterns is the following lemma which shows that they are ubiquitous in pairings of longer $\bfa$-APs.

\begin{lemma}
\label{claim:pairing}
In any pairing $f$ of $\bfa$ one of the following holds:
\begin{enumerate}
    \item there exist $1\leq i_1<i_2<i_3<i_4\leq k$ with $f(i_1)=i_3$ and $f(i_2)=i_4$; or 
    \item there exist $1\leq i_1<i_2<i_3<i_4\leq k$ with $f(i_1)=i_4$ and $f(i_2)=i_3$ and $a_{i_1}+a_{i_4}\ne a_{i_2}+a_{i_3}$; or
    \item $\bfa$ is symmetric and $f$ is the symmetric pairing $f(i)=k+1-i$.
\end{enumerate}
\end{lemma}
\begin{proof}
If there exists some $2\le i < f(1)$ with $f(i) > f(1)$, then the tuple $(1, i, f(1), f(i))$ shows we are in case (1).
    
Now assume that no such $i$ exists. If for any $2\le i < f(1)$, we have $a_i+a_{f(i)}\ne a_1+a_{f(1)}$, then $(1, \min(i, f(i)), \max(i, f(i)), f(1))$ shows we are in case (2). 

Next assume that for all $2\le i < f(1)$, $a_i+a_{f(i)}= a_1+a_{f(1)}$. From this we see that $|a_1-a_i| | a_1 - a_{f(i)}| = |a_{f(1)} - a_i| | a_{f(1)} - a_{f(i)}|$ for any $2\le i < f(1)$.  For $i > f(1)$, we have $|a_1-a_i| > |a_{f(1)} - a_i|$. Thus
\begin{align*}
|c_1|&=\prod_{\{i,f(i)\}:i<f(1)}|a_1-a_i||a_1-a_{f(i)}|\prod_{i>f(1)}|a_1-a_i|\\
&\geq\prod_{\{i,f(i)\}:i<f(1)}|a_{f(1)}-a_i||a_{f(1)}-a_{f(i)}|\prod_{i>f(1)}|a_{f(1)}-a_i|=|c_{f(1)}|.
\end{align*}
Note that the inequality is strict unless the product over $i>f(1)$ is empty. Since we know $|c_1|=|c_{f(1)}|$ this implies that the product is indeed empty, so $f(1)=k$.

Thus if we are not in case (1) or (2), we have $f(1)=k$ and $a_1+a_k=a_i+a_{f(i)}$ for all $i\in[k]$. Thus $\bfa$ is symmetric and we see that $a_{f(1)}>a_{f(2)}>a_{f(3)}>\cdots$ implying that $f(i)=k+i-1$ for all $i$, so we are in case (3).
\end{proof}

\begin{lemma}
\label{lem:gen-a-ap-coloring}
Let $\bfa=(a_1,\ldots,a_k)$ have increasing integer coordinates. There is a constant $C_\bfa$ so that the following holds. If $k$ is odd or $k$ is even and $\bfa$ is asymmetric, then for every $r$ sufficiently large there exist $N=\Theta(r^{C_{\bfa}\log r})$ and an $r$-coloring $\Phi$ of $\TT$ such that 
\begin{equation}
\label{eq:a-bin-density}
\Pr_{x,y\in\TT}(x+a_1y,\ldots,x+a_ky\text{ forms an $\bfa$-binomial pattern in }\Phi)\lesssim_{\bfa}\frac1N.
\end{equation}
Furthermore, if $k$ is even and $\bfa$ is symmetric, if there exists a $t$-coloring of $\Z/N\Z$ that avoids symmetrically colored $\bfa$-APs, then there exists an $rt$-coloring $\Phi$ of $\TT$ that also satisfies \cref{eq:a-bin-density}.
\end{lemma}

\begin{proof}
Without loss of generality, let us assume that $a_1>0$. (This is allowed since all the definitions are invariant under the translation $(a_1,\ldots,a_k)\leadsto(1,a_2-a_1+1,\ldots,a_k-a_1+1)$.) Let $N=\Theta(r^{C_{\bfa}\log r})$ be an appropriately chosen parameter. We first construct a coloring $\phi$ of $\Z/N\Z$ that has a stronger property than avoiding $\bfa$-binomial patterns.

By \cref{lem:avoid-pattern}, there exists a coloring $\psi\colon\Z/N\Z\to[\floor{r^{1/2}}]$ that avoids monochromatic $a_k$-patterns. When $k$ is odd, let $\phi=\psi$, which already avoids $\bfa$-binomial patterns.

When $k$ is even by \cref{lem:mod-behrend-coloring} there is a coloring $\chi\colon\Z/N\Z\to[\floor{r^{1/2}}]$ that avoids $(a_{i_1},a_{i_2},a_{i_3},a_{i_4})$-ABAB patterns for all $1\leq i_1<i_2<i_3<i_4\leq k$ and also avoids $(a_{i_1},a_{i_2},a_{i_3},a_{i_4})$-ABBA patterns for all $1\leq i_1<i_2<i_3<i_4\leq k$ with $a_{i_1}+a_{i_4}\neq a_{i_2}+a_{i_3}$. When $k$ is even and $\bfa$ is asymmetric, let $\phi$ be the product coloring of $\psi$ and $\chi$.

Finally if $k$ is even and $\bfa$ is symmetric, let $\omega\colon\Z/N\Z\to[t]$ be the hypothesized coloring that avoids symmetrically colored $\bfa$-APs. Then let $\phi$ be the product coloring of $\psi$, $\chi$, and $\omega$.

In each case we have a $u$-coloring $\phi\colon \Z/N\Z\to[u]$. By \cref{claim:pairing}, we see that $\phi$ avoids $\bfa$-binomial patterns. In particular, if $k$ is even and $\bfa$ is asymmetric we cannot be in case (3), so we see that $\phi$ avoids $\bfa$-binomial patterns. Similarly, if $k$ is even and $\bfa$ is symmetric, by definition $\phi$ avoids symmetrically colored $\bfa$-APs and thus avoids all $\bfa$-binomial patterns.

Now set $m=(a_k-a_1+1)!$ and define $\Phi\colon\TT\to[mu]$ by interlacing $m$ copies of $\phi$, each using disjoint sets of colors. More precisely, \[\Phi(x)=\phi(\floor{Nx})+u(\lfloor mNx\rfloor\pmod m).\]

To analyze $\Phi$, suppose that $x+a_1y,\ldots,x+a_ky\in\TT$ is an $\bfa$-binomial pattern in $\Phi$. Define $n_i\in \Z/mN\Z$ so that $n_i\equiv \floor{mN(x+iy)}$ and define $m_i\in[m]$ by $m_i\equiv \floor{mN(x+iy)}\pmod m$. Then by definition we have $\Phi(x+a_iy)=\Phi(x+a_jy)$ if and only if $m_i=m_j$ and $\phi(\floor{n_i/m})=\phi(\floor{n_j/m})$. We wish to show that for any such $\bfa$-binomial pattern, there must be two distinct indices $i\neq i'\in[k]$ such that $n_i=n_{i'}$.

We know that $n_{i+1}-n_i\in\{\floor{mNy},\ceil{mNy}\}$ for each $i$, so the $n_i$ form an AP with jumps of size 1. Suppose we have the second type of $\bfa$-binomial pattern; that is, there is some $\ell\geq 3$ and $I=\{i_1<i_2<\cdots<i_\ell\}\subset[k]$ such that $\{x+a_iy\}_{i\in I}$ are assigned the same color under $\Phi$. This implies that $m_{a_{i_1}}=m_{a_{i_2}}=\cdots=m_{a_{i_\ell}}$, in particular, that $n_{a_{i_1}}\equiv n_{a_{i_\ell}}\pmod m$. Recalling that we chose $m=(a_k-a_1+1)!$, we see that \cref{lem:error-correct} implies that $n_{a_{i_1}},n_{{a_{i_1}}+1},\ldots, n_{a_{i_\ell}}$ is an AP. Pick $n,d\in \Z/mN\Z$ so that $n+id\equiv n_i\pmod{mN}$ for all ${a_{i_1}}\leq i\leq {a_{i_\ell}}$. We see that $(n+a_1d,\ldots,n+{a_k}d)$ is an $\bfa$-binomial pattern in $\phi$ since $n+a_jd=n_{a_j}$ for $i_1\leq j\leq i_\ell$. By our construction of $\phi$, this can only occur if $n_{a_{i_1}}=n_{{a_{i_1}}+1}=\cdots=n_{a_{i_\ell}}$.

Similarly, suppose we have the first type of $\bfa$-binomial pattern; that is, there is a pairing $f$ of $\bfa$ such that $\Phi(x+a_iy)=\Phi(x+a_{f(i)}y)$ for all $i\in[k]$. We apply \cref{claim:pairing} to analyze the structure of $f$.

In case (1), we have $1\leq i_1<i_2<i_3<i_4\leq k$ with $f(i_1)=i_3$ and $f(i_2)=i_4$. By an identical application of \cref{lem:error-correct} we find that $n_{a_{i_1}},n_{a_{i_1}+1},\ldots,n_{a_{i_4}}$ is an AP. Therefore $n_{a_{i_1}},n_{a_{i_2}},n_{a_{i_3}},n_{a_{i_4}}$ form an $(a_{i_1},a_{i_2},a_{i_3},a_{i_4})$-ABAB pattern in $\phi$. By our construction of $\phi$, this can only occur if $n_{a_{i_1}}=n_{a_{i_1}+1}=\cdots=n_{a_{i_4}}$.

In case (2), we similarly find that $n_{a_{i_1}},n_{a_{i_2}},n_{a_{i_3}},n_{a_{i_4}}$ form an asymmetric $(a_{i_1},a_{i_2},a_{i_3},a_{i_4})$-ABBA pattern in $\phi$ and thus get the same conclusion.

Finally in case (3), by the fact that $f(1)=k$, we see by \cref{lem:error-correct}, that $n_{a_1},n_{a_1+1},\ldots,n_{a_k}$ is an AP and thus $n_{a_1},n_{a_2},\ldots,n_{a_k}$ forms a symmetrically colored $\bfa$-AP in $\phi$, again getting the desired conclusion.

Therefore we have shown that if $x+a_1y,\ldots,x+a_ky\in\TT$ form an $\bfa$-binomial pattern in $\Phi$, at least two of the terms say $x+a_iy,x+a_{i'}y$ must lie in the same interval of length $1/(mN)$. The probability of this occurring is $\lesssim_{\bfa} 1/N$ since there are $O_k(1)$ choices for $i,i'$, a $1/mN$ probability for $x+a_{i'}y$ to lie in the same interval as $x+a_iy$, and $O_\bfa(1)$ choices for $x,y$ once $x+a_iy, x+a_{i'}y$ are fixed. This proves the desired property of $\Phi$.
\end{proof}

Now \cref{thm:general_reduction,thm:general_pattern} follow from the same proofs as \cref{thm:kap-cons-odd,thm:kap-reduction-even}, using \cref{lem:gen-a-ap-construction} in the place of \cref{lem:gen-k-ap-construction} and using \cref{lem:gen-a-ap-coloring} as the coloring input.

\appendix

\section{Equidistribution and convergence}
\label{sec:equi-conv}

Here we prove \cref{prop:conv}, showing that the various linear form statistics on $f_N(n) = F(n/N \bmod 1, n^{k-2}/N \bmod 1)$ converge to the associated linear form statistics on $F$ as $N \to \infty$. 
As this is a standard exercise on Weyl equidistribution, we will be somewhat sketchy here.

By a standard square differencing and degree reduction (analogous to Gauss sums), we have the following result, whose proof we omit. Here we use the standard notation $e(x) = e^{2\pi i x}$ for $x \in \RR$.

\begin{lemma}[Complete Weyl sums] \label{lem:weyl}
Let $P(x_1, \dots, x_s)$ be a fixed non-constant polynomial with integer coefficients. Then, as $N \to \infty$,
\[
\frac{1}{N^s} \sum_{n_1, \dots, n_s =1}^N e\paren{\frac{P(n_1, \dots, n_s)}{N}} \to 0.
\]
\end{lemma}

Let us first illustrate the proof of $\Lambda_k(f_N) \to \wt\Lambda_k(F)$. 
We consider a multilinear generalization.
Suppose we have smooth $F_1, \dots, F_k \colon \TT^2 \to [0,1]$. For each $j = 1, \dots, k$ and $N$, define $f_{i,N} \colon \ZZ/N\ZZ \to [0,1]$ by $f_{i, N}(n) \coloneqq F_i(n/N, n^{k-2}/N)$. We will show that 
\begin{equation}\label{eq:conv-kap-multilinear}
\Lambda_k(f_{1,N}, \dots, f_{k, N}) \to \wt \Lambda_k(F_1, \dots, F_k), \quad \text{as } N \to \infty
\end{equation}
where  $\wt\Lambda_k(F_1,\ldots,F_k)$ is the natural multilinear generalization of $\wt\Lambda_k$ defined by
\[\wt\Lambda_k(F_1,\ldots,F_k) \coloneqq \EE_{(x_1, \dots, x_k, y_1, \dots, y_k) \in V} F_1(x_1, y_1) \cdots F_k(x_k, y_k)\]
where $V$ is the subset of $\TT^{2k}$ defined by all points $(x_1, \dots, x_k, y_1, \dots, y_k)$ satisfying
\[
x_2 - x_1 = \cdots = x_k - x_{k-1}
\]
and
\[
\sum_{i=0}^{k-1} (-1)^i\binom{k-1}{i} y_i = 0.
\]

To prove \cref{prop:conv}, it suffices to prove it assuming that $F_j$ is smooth, since we can approximate each $F_j$ from above and below by smooth $[0,1]$-valued functions with arbitrary small $L^1$ approximation error.

Consider the Fourier series expansion 
\[
F_j(x,y) = \sum_{r,s \in \ZZ} \wh{F_{j}}(r,s) e(rx + sy).
\]
Since $F_j$ is smooth, its Fourier coefficients decay faster than any polynomial, i.e., $\abss{\wh{F_j}(r,s)} \lesssim_A (1 + \abs{r} + \abs{s})^{-A}$ for every $A> 0$. 
So, due to absolute convergence, we can check \cref{eq:conv-kap-multilinear} term-by-term after expanding using the Fourier series. In other words, it suffices to check \cref{eq:conv-kap-multilinear} when each $F_j$ has the form $F_j(x,y) = e(r_jx + s_jy)$ for some $r_j, s_j \in \ZZ$.
In this case, we have
\begin{align*}
\text{LHS of \cref{eq:conv-kap-multilinear}}
&= \EE_{n_0, n_1 \in \ZZ/N\ZZ} \prod_{j=1}^k f_{j,N}(n_0 + jn_1)
\\
&= \EE_{n_0, n_1 \in \ZZ/N\ZZ} \prod_{j=1}^k F_j\paren{ \frac{n_0 + jn_1}{N}, \frac{(n_0 + jn_1)^{k-2}}{N}}
\\
&= 
\EE_{n_0, n_1 \in \ZZ/N\ZZ} 
e\paren{\frac{P(n_0, n_1)}{N}}
\end{align*}
where
\[
P(n_0, n_1) = \sum_{j=1}^k \paren{r_j(n_0 + jn_1) + s_j(n_0 + jn_1)^{k-2}}.
\]
If $P$ is the zero polynomial, then the LHS of \cref{eq:conv-kap-multilinear} is one.
Otherwise, by \cref{lem:weyl}, the LHS of \cref{eq:conv-kap-multilinear} converges to zero as $N \to \infty$.

Note that $P$ is the zero polynomial if and only if $\sum_{j=1}^k r_j=\sum_{j=1}^k jr_j=0$ and 
\begin{equation}
\label{eq:vandermonde-system}
\sum_{j=1}^k s_jj^t=0 \text{ for each }0\leq t\leq k-2.
\end{equation}
It is not hard to see that $s_j = (-1)^j \binom{k-1}{j-1}\lambda$ satisfies the system \cref{eq:vandermonde-system} for each $\lambda\in\R$. (See \cref{lem:a-binom-eq} for the proof of a more general identity.) Furthermore these are the only solutions to \cref{eq:vandermonde-system} since the matrix $(j^t)_{j\in[k], 0\leq t\leq k-2}$ is the Vandermonde matrix which is full rank.

On the other hand, 
\begin{align*}
\text{RHS of \cref{eq:conv-kap-multilinear}}
&= \EE_{x_0, x_1, y_1, \dots, y_k \in \TT \colon \sum_{j=1}^k (-1)^j \binom{k-1}{j-1} y_j = 0} \prod_{j=1}^k F(x_0 + jx_1, y_j)
\\
&= \EE_{x_0, x_1, y_1, \dots, y_k \in \TT \colon  \sum_{j=1}^k (-1)^j \binom{k-1}{j-1} y_j = 0} 
\prod_{j=1}^k e(r_j(x_0 + jx_1) + s_jy_j),
\end{align*}
which is equal to one if
\[\sum_{j=1}^k \paren{r_j(x_0+jx_1)+s_jy_j} = 0 \text{ for all } x_0, x_1, y_1, \dots, y_k \in \TT \text{ such that } \sum_{j=1}^k (-1)^j \binom{k-1}{j-1} y_j = 0 \]
and is equal to zero otherwise. The above is true if and only if there is some $\lambda \in \RR$ such that
\[
\sum_{j=1}^k r_j = 0,
\quad 
\sum_{j=1}^k j r_j = 0,
\text{ and } 
s_j = (-1)^j \binom{k-1}{j-1}\lambda \text{ for each } 1 \le j \le n,
\]
This condition is precisely equivalent to $P$ being the zero polynomial, meaning that \cref{eq:conv-kap-multilinear} holds when the $F_1, \dots, F_k$ are exponential phases. The general result follows from summing the terms of the Fourier series expansion of each $F_j$.

The proof for $\norm{f - \alpha}_{U^{k-2}}$ is similar but cumbersome to write out. Let us just illustrate the proof in the representative case of $k = 4$. 
It suffices to prove that, given Riemann integrable functions $F_{00}, F_{01}, F_{10}, F_{00} \colon \TT^2 \to [0,1]$, defining, for each $j \in \set{00,01,10,11}$, $f_{j,N} \colon \ZZ/N\ZZ \to [0,1]$ by $f_{j,N}(n) = F_{j}(n/N, n^2/N)$, and letting $F_j^1$ be the first marginal of $F_j$, defined by $F_j(x):=\int_\TT  F(x,y)\,\text dy$ we have, as $N \to \infty$,
\begin{multline} \label{eq:conv-u2-multlinear}
\EE_{n_0, n_1, n_2 \in \ZZ/N\ZZ} f_{00} (n_0) f_{10}(n_0 + n_1) f_{01}(n_0 + n_2) f_{11}(n_0 + n_1 + n_2) 
\\
\to 
\EE_{x_0,x_1,x_2 \in \TT} F_{00}^1(x_0) F_{10}^1(x_0+x_1) F_{01}^1(x_0+x_2) F_{11}^1(x_0+x_1+x_2).
\end{multline}
As earlier, by approximating with smooth functions, and then considering Fourier series, it suffices to prove this convergence when each $F_j$ is an exponential phase $F_j(x,y) = e(r_jx + s_jy)$, in which case
\begin{align*}
\text{LHS of \cref{eq:conv-u2-multlinear}} = \EE_{n_0, n_1, n_2} e\paren{\frac{P_1(n_0, n_1, n_2) + P_2(n_0, n_1, n_2)}{N}}
\end{align*}
where
\[
P_1(n_0, n_1, n_2) = r_{00} n_0 + r_{10}(n_0 + n_1) + r_{01} (n_0 + n_2) + r_{11} (n_0 + n_1 + n_2)
\]
and
\[
P_2(n_0, n_1, n_2) = s_{00} n_0^2 + s_{10}(n_0 + n_1)^2 + s_{01} (n_0 + n_2)^2 + s_{11} (n_0 + n_1 + n_2)^2
\]
Note that $P_2$ is identically zero if and only if $s_{00} = s_{10} = s_{01} = s_{11} = 0$.

If $(s_{00}, s_{10}, s_{01}, s_{11}) \ne (0,0,0,0)$, then by \cref{lem:weyl}, the LHS of \cref{eq:conv-u2-multlinear} converges to zero, but we also have $F_j^1$ being the zero function whenever $s_j \ne 0$. 
So \cref{eq:conv-u2-multlinear} holds in this case.

On the other hand, if $(s_{00}, s_{10}, s_{01}, s_{11}) =  (0,0,0,0)$, then only linear exponential phases remain, and it is also straightforward to check that \cref{eq:conv-u2-multlinear} holds. Basically, the limit is one if $P_1$ is identically zero, and zero otherwise.

This completes our sketch of the proof of \cref{prop:conv}. The remaining details are standard and straightforward.

\end{document}